\documentclass[12pt,reqno]{amsart}

\usepackage{amssymb,amsmath, amsfonts, latexsym}
\usepackage{enumerate}
\newtheorem{thm}{Theorem}[section]
\newtheorem{cor}[thm]{Corollary}
\newtheorem{lem}[thm]{Lemma}
\newtheorem{prop}[thm]{Proposition}
\newtheorem{example}[thm]{Example}
\newtheorem{remarks}[thm]{Remarks}
\newtheorem{defn}[thm]{Definition}

 \numberwithin{equation}{section}

\title[Transportation-information inequalities]{Transportation-information inequalities for Markov processes}

\author{Arnaud Guillin}
\address{Arnaud Guillin. Ecole Centrale de Marseille et LATP,
Centre de Math\'ematiques et Informatique. Technop\^ole de
Ch\^ateau-Gombert. 13453 Marseille, France}
\email{guillin@cmi.univ-mrs.fr}

\author{Christian L\'eonard}
\address{Christian L\'eonard. Modal-X, Universit\'e Paris 10. B\^atiment G, 200 avenue de la R\'epublique.92001 Nanterre, France}
\address{Christian L\'eonard. CMAP, \'Ecole Polytechnique. 91128 Palaiseau, France}
\email{christian.leonard@polytechnique.fr}

\author{Liming Wu}
\address{Liming Wu. Laboratoire de Math\'ematiques Appliqu\'ees, CNRS-UMR
6620, Universit\'e Blaise Pascal, 63177 Aubi\`ere, France. And
Department of Mathematics, Wuhan University, 430072 Hubei, China}
\email{Li-Ming.Wu@math.univ-bpclermont.fr}

\author{Nian Yao}
\address{Nian Yao. Department of Mathematics, Wuhan University, 430072 Hubei, China}

\date{May 02, 2007}

\newcommand{\dd}{\mathbb{D}}
\newcommand{\ee}{\mathbb{E}}

\newcommand{\nn}{\mathbb{N}}
\newcommand{\rr}{\mathbb{R}}
\newcommand{\pp}{\mathbb{P}}

\newcommand{\Var}{\mathrm{Var}}

\def\AA{\mathcal A}
\def\BB{\mathcal B}
\def\CC{\mathcal C}
\def\DD{\mathcal D}
\def\FF{\mathcal F}

\def\EE{\mathcal E}

\def\LL{\mathcal L}
\def\MM{\mathcal M}
\def\NN{\mathcal N}

\def\XX{\mathcal X}

\def\SG{c_{\mathrm{P}}}
\def\LS{c_{\mathrm{LS}}}
\def\WI{c_{\mathrm{W_1I}}}
\def\WWI{c_{\mathrm{W_2I}}}

\def\vep{\varepsilon}
\def\<{\langle}
\def\>{\rangle}

\def\brmk{\begin{remarks}}
\def\nrmk{\end{remarks}}

\def\bexa{\begin{example}}
\def\nexa{\end{example}}

\def\bcor{\begin{cor}}
\def\ncor{\end{cor}}

\def\bexe{\begin{exe}}
\def\nexe{\end{exe}}

\def\bprf{\begin{proof}}
\def\nprf{\end{proof}}

\def\dsp{\displaystyle}

\def\bdes{\begin{description}}
\def\ndes{\end{description}}

\begin{document}

\begin{abstract}
In this paper, one investigates the following type of
transportation-information $T_cI$ inequalities:
$\alpha(T_c(\nu,\mu))\le I(\nu|\mu)$ for all probability measures
$\nu$ on some metric space $(\XX, d)$, where $\mu$ is a given
probability measure, $T_c(\nu,\mu)$ is the transportation cost
from $\nu$ to $\mu$ with respect to some cost function $c(x,y)$ on
$\XX^2$,  $I(\nu|\mu)$ is the Fisher-Donsker-Varadhan information
of $\nu$ with respect to $\mu$ and $\alpha: [0,\infty)\to
[0,\infty]$ is some left continuous increasing function. Using
large deviation techniques, it is shown that  $T_cI$ is equivalent
to some concentration inequality for the occupation measure of a
$\mu$-reversible ergodic Markov process related to $I(\cdot|\mu)$,
a counterpart of the characterizations of transportation-entropy
inequalities, recently obtained by Gozlan and L\'eonard in the
i.i.d.\! case  \cite{GL}. Tensorization properties of $T_cI$ are
also derived.
\\
Let $d$ be a metric. One denotes $W_1I:=T_dI$ and $W_2I:=T_{d^2}I$
the transportation-information inequalities associated with the
metric cost $c=d$ and the quadratic cost $c=d^2.$
\\
It is proved that $W_2I$ is stronger than Poincar\'e inequality,
weaker than log-Sobolev inequality, and equivalent to it when
Bakry-Emery's curvature is bounded from below.
\\
For the trivial metric cost $d,$ one establishes the sharp
transportation-information inequality $W_1I$ in terms of the
spectral gap. In particular, a Hoeffding type concentration
inequality for Markov processes is derived and one shows that
$W_1I$ implies a Poincar\'e inequality.
\\
For a general metric cost $d,$ it is established that the spectral
gap in the space of Lipschitz functions of the Markov diffusion
process implies  $W_1I.$ A sharp estimate of the constant is
obtained for general one-dimensional diffusion processes. Finally,
a Lyapunov function condition for $W_1I$ is proposed. It may be
applied to a wide class of examples; some examples are worked out.
\end{abstract}

\maketitle

\tableofcontents

\section{Introduction}

Let $M_1(\XX)$ be the space of all probability measures on some
complete separable metric space $(\XX,d)$ and consider the {\it
cost function} $c(x,y):\XX^2\to [0,+\infty]$ with $c(x,x)=0$ (for
all $x\in \XX$), which is lower semicontinuous on $\XX^2$. Given
$\mu,\nu\in M_1(\XX)$,  the {\it transportation cost} $T_c(\nu,
\mu)$ from $\nu$ to $\mu$ with respect to the cost function $c$ is
defined by

\begin{equation}\label{cost}
T_c(\nu,\mu)=\inf_{\pi\in M_1(\XX^2):\pi_0=\nu, \pi_1=\mu}
\iint_{\XX^2} c(x,y)\, \pi(dx, dy)
\end{equation}
where $\pi_0(dx)=\pi(dx\times \XX),\ \pi_1(dy)=\pi(\XX\times dy)$
are the marginal distributions of $\pi$. When $c(x,y)=d^p(x,y)$
where $p\ge 1$, $(T_c(\nu,\mu))^{1/p}=W_p(\nu, \mu)$ is the $L^p$-
Wasserstein distance between $\nu$ and $\mu$.
\\
The relative entropy (or Kullback information) of $\nu$ with
respect to $\mu$ is given by

\begin{equation}\label{H}
H(\nu|\mu):= \begin{cases}
\dsp\int_\XX f\log f \,d\mu, \ \ &\text{if } \ \nu\ll\mu \ \text{and} \ f:=\frac{d\nu}{d\mu}\\
+\infty, &\text{otherwise.}
\end{cases}
\end{equation} The usual transportation inequalities for a given $\mu\in
M_1(\XX)$, introduced by K.~Marton \cite{Mar96, Mar97} and
M.~Talagrand \cite{Tal96a}, compare the Wasserstein metric
$W_p(\nu, \mu)$ with the relative entropy $H(\nu|\mu)$. The
following extension of these inequalities:
\begin{equation}
    \alpha(T_c(\nu,\mu)) \le H(\nu|\mu), \ \forall \nu\in M_1(\XX),
    \tag{$T_cH$}
\end{equation}
has recently been proposed and developed by Gozlan and L\'eonard
\cite{GL}. Here $\alpha: [0,\infty)\to [0,+\infty]$ is some left
continuous and increasing function with $\alpha(0)=0$.
\\
 Let us denote
\begin{equation}\label{b1}
    \alpha^\circledast(\lambda):=\sup_{r\ge 0}(\lambda r -
\alpha(r))
\end{equation}
the monotone conjugate of $\alpha.$ With $\alpha$ as above, one
sees that $\alpha^\circledast$ is the restriction to $[0,\infty)$
of the usual convex conjugate
$\tilde{\alpha}^*(\lambda)=\sup_{r\in\mathbb{R}}(\lambda r -
\tilde{\alpha}(r))$ of
$\tilde{\alpha}(r)=\mathbf{1}_{r\ge0}\alpha(r),$ $r\in
\mathbb{R}.$ We also denote $\mu(v):=\int_\XX v d\mu.$

\par\medskip\noindent\textbf{Notation.}\ In the special case
$c(x,y)=d^p(x,y)$ where $p\ge1$ and $d$ is a metric,
$T_c(\nu,\mu)=W_p(\nu, \mu)^{p}.$ We shall use the notation $W_pH$
instead of $T_{d^p}H.$ In particular, $W_1H$ stands for $T_dH.$

As an extension of the Bobkov-G\"otze criterion \cite{BG99}, we
have

\begin{thm}[Gozlan-L\'eonard \cite{GL}]\label{thm-GL}  Let $(X_n)_{n\in
\nn}$ be a sequence of $\XX$-valued i.i.d.\! random variables with
common law $\mu$ and $\alpha$ be moreover convex. Then the
following properties are equivalent:
\begin{enumerate}
\item[(a)] The transportation inequality $T_cH$ holds;

\item[(b)] For any couple of bounded and measurable functions $u,
v : \XX\to \rr$ such that $u(x)-v(y)\le c(x,y)$ over $\XX^2$,
$$
\log \int_\XX e^{\lambda u} d\mu\le \lambda \mu(v) +
\alpha^\circledast (\lambda), \ \forall \lambda\ge 0;
$$

\item[(c)] For all $n\ge 1$ and $r>0$ and for any couple of
bounded and measurable functions $u, v : \XX\to \rr$ such that
$u(x)-v(y)\le c(x,y)$ over $\XX^2$,  the following concentration
inequality holds

$$
\pp\left( \frac 1n\sum_{k=1}^n u(X_k)\ge \mu(v) + r\right)\le
e^{-n \alpha(r)};
$$

\item[(c')] The following large deviation upper bound holds for
any couple of bounded and measurable functions $u, v : \XX\to \rr$
such that $u(x)-v(y)\le c(x,y)$ over $\XX^2$,

$$
\limsup_{n\to \infty} \frac 1n\log  \pp\left( \frac 1n\sum_{k=1}^n
u(X_k)\ge \mu(v) + r\right)\le -\alpha(r), \ \forall r\ge0.
$$
\end{enumerate}
\end{thm}

\subsection*{The main purpose of this paper}
In this paper, instead of the transportation-entropy inequality
$T_cH,$ one investigates the following transportation-information
inequality
\begin{equation}
    \alpha(T_c(\nu,\mu))\le I(\nu|\mu), \ \forall \nu\in M_1(\XX)
    \tag{$T_cI$}
\end{equation}
for some given probability measure $\mu$. Here $I(\nu|\mu)$ is the
Fisher-Donsker-Varadhan information of $\nu$ with respect to $\mu$
\begin{equation}\label{FDV} I(\nu|\mu)=
  \begin{cases}
    \EE(\sqrt f, \sqrt f) & \text{if } \ \nu=f\mu,\ \sqrt f\in \dd(\EE) \\
    +\infty & \text{otherwise}
  \end{cases}
 \end{equation}
associated with the Dirichlet form $\EE$ on $L^2(\mu)$ with domain
$\dd(\EE).$

\par\medskip\noindent\textbf{Notation.}\
 In the special case where  $c(x,y)=d^p(x,y),$ we use the notation
$W_pI$ instead of $T_{d^p}I.$ In particular, $W_1I$ stands for
$T_dI.$

\subsection*{Organization of the paper}
This paper is organized as follows. In the next section we
characterize $T_cI$ by means of concentration inequalities for the
empirical means $L_t(u)=\frac{1}{t}\int_0^t u(X_s)\,ds$ of
observables $u$ along the symmetric Markov process $(X_t)$
associated with the Dirichlet form $\EE$, extending Theorem
\ref{thm-GL} from i.i.d.\! sequences to time-continuous Markov
processes. The method of proof is borrowed from Gozlan and
L\'eonard \cite{GL} who proved Theorem \ref{thm-GL}  by means of
large deviations of the empirical measure of an i.i.d.\! sequence.
In the present paper, it relies on the large deviations of the
occupation measure of $(X_t).$ The tensorization of $T_cI$ is
proved, and the relations between $W_2I$, Poincar\'e and
log-Sobolev are exhibited with the help of Otto-Villani
\cite{OVill00}.

In Section 3, we prove $W_1I$ for the trivial metric
$d(x,y)=\mathbf{1}_{x\ne y}$ with the sharp constant in terms of
the spectral gap and derive a sharp Hoeffding concentration
inequality for Markov processes. Furthermore, we also prove that
$W_1I$ implies the existence of a positive spectral gap in the
symmetric and uniform positive improving case by means of a result
of L.~Wu \cite{p-Wu01}.

For a general metric, using Lyons-Meyer-Zheng forward-backward
martingale decomposition, we obtain in Section 4 sharp $W_1I$
under the spectral gap existence of the Markov diffusion process
in the space of Lipschitz functions, and an explicit and sharp
constant is provided for one-dimensional diffusions.

Finally in Section 5 we propose a practical Lyapunov condition for
$W_1I$ (or a more general $T_\Phi I$), which, although not
providing the sharp constant, provides a good order.

\subsection*{About the literature}
Let us give some historical notes on the usual transportation
inequality $W_pH$. K.~Marton \cite{Mar96} first noticed that
$W_1H$ implies the concentration inequality for $\mu$ by a very
elementary and neat argument, and she established $W_1H$ for the
law of a Dobrushin-contractive Markov chain in \cite{Mar97}.
M.~Talagrand \cite{Tal96a} established $W_2H$ for the Gaussian
measure $\mu$ with $\alpha(r)=r/2C$ and provided the sharp
constant $C$ (this particular case of $T_cH$ is often called
Talagrand's transportation inequality). Bobkov and G\"otze
\cite{BG99} obtained the characterization of $W_pH$ in Theorem
\ref{thm-GL} with [$p=1,$ $\alpha$ quadratic] and [$p=2,$ $\alpha$
linear]. Otto and Villani \cite{OVill00} proved that the
log-Sobolev inequality is stronger than Talagrand's transportation
inequality and presented a differential geometrical point of view
on $M_1(\XX)$ equipped with the $W_2$-metric. Bobkov, Gentil and
Ledoux \cite{BGL01} shed light on a profound relation between
Talagrand's transportation inequality, log-Sobolev inequality,
inf-convolution and some Hamilton-Jacobi equation. Djellout,
Guillin and Wu \cite{DGW03} obtained a necessary and sufficient
condition for $W_1H$ with a quadratic $\alpha$  by means of the
Gaussian integrability of $d(x,x_0)$ under $\mu$, and gave a
direct proof of Talagrand's transportation inequality for the law
of a diffusion process by means of Girsanov's formula, without
appealing to log-Sobolev inequality. Bolley and Villani
\cite{BV05} and later Gozlan and L\'eonard \cite{GL} refined the
result of \cite{DGW03} under a Gaussian integrability condition.
Cattiaux and Guillin \cite{p-CG04} constructed  the first example
for which Talagrand's transportation inequality holds but not
log-Sobolev inequality, and Gozlan \cite{p-Goz06} found a
necessary and sufficient condition for Talagrand's transportation
inequality with $\mu(dx)=e^{-V(x)} dx$ on $\rr$ when the
Bakry-Emery curvature $V^{\prime\prime}$ is lower bounded.
Otto-Villani's differential geometrical point of view on
$M_1(\XX)$ equipped with the $W_2$-metric is very fruitful, as
developed by the recent works of Sturm \cite{St06a, St06b} and
Lott and Villani \cite{p-LV06}. The reader is referred to the
textbooks by Ledoux \cite{Led01} and Villani \cite{Vill03,
p-Vill05} for further references pertaining to this very active
field.

\subsection*{Convention and notation}
Throughout this paper either $(\XX,d)$ is a complete separable
metric space with the associated Borel $\sigma$-field $\BB.$
\begin{enumerate}[-]
    \item The space of all real  bounded and $\BB$-measurable functions is
denoted by $b\BB$.
    \item The functions to be considered later are assumed to
    be measurable without warning.
    \item For $\mu,\nu\in M_1(\XX)$,
$\|\nu-\mu\|_{TV}:=\sup_{u: |u|\le 1}\int u\,d(\nu-\mu)$ is the
total variation norm.
    \item Throughout this paper a cost function $c$ is a nonnegative lower
semicontinuous function on $\XX^2$ such that $c(x,x)=0$ for all
$x\in \XX$.
\end{enumerate}

\section{General results on $T_cI$}

\subsection{Markov processes, Fisher-Donsker-Varadhan information and Feynman-Kac semigroup}
The main probabilistic object to be considered in this paper is an
$\XX$-valued time-continuous Markov process $(\Omega, \FF,
(X_t)_{t\ge0}, (\pp_x)_{x\in \XX})$  with an invariant probability
measure $\mu.$ The transition semigroup is denoted
$(P_t)_{t\ge0}.$

\subsubsection*{Assumption: Ergodicity}
It is assumed that the invariant probability measure $\mu$ is
ergodic: if $f\in b\BB$ satisfies $P_tf=f,$ $\mu$-a.e. for all
$t\ge0$, then $f$ is constant $\mu$-a.e. Denoting
$\pp_\beta(\cdot):=\int_\XX \pp_x(\cdot)\,\beta(dx)$ for any
initial probability measure $\beta,$ the previous condition on
$\mu$ amounts to stating that $((X_t)_{t\ge0}, \pp_\mu)$ is a
stationary ergodic process.

\subsubsection*{Assumption: Closability of the Dirichlet form}
It is assumed that $(P_t)$ is strongly continuous on
$L^2(\mu):=L^2(\XX,\BB,\mu)$. Let $\LL$ be its generator with
domain $\dd_2(\LL)$ on $L^2(\mu)$. It is also assumed that
$$
\EE(g,g):=\<-\LL g, g\>_{\mu}, \ g\in \dd_2(\LL)
$$
is closable in $L^2(\mu).$ Its closure which is denoted by $(\EE,
\dd(\EE))$ is a Dirichlet form: the symmetrized Dirichlet form
associated with the Markov process $(X_t)$ (or $(P_t)$). Notice
that $(\EE, \dd(\EE))$ corresponds to a self-adjoint generator
$\LL^{\sigma}$ (formally $\LL^\sigma=(\LL +\LL^*)/2$), and
$P_t^\sigma=e^{t\LL^\sigma}$ is the symmetrized Markov semigroup
of $(P_t)$. When $P_t$ is symmetric on $L^2(\mu)$, the above
closability assumption is always satisfied and the domain
$\dd(\EE)$ of the Dirichlet form coincides with the domain
$\dd_2(\sqrt{-\LL})$ in $L^2(\mu)$.

These above assumptions of ergodicity and closability of the
Dirichlet form prevail for the whole paper.

\subsubsection*{Fisher-Donsker-Varadhan information} The following
definition is motivated by standard large deviation results.

\begin{defn} Given the Dirichlet form $\EE$ with domain $\dd(\EE)$ on
$L^2(\mu)$, the Fisher-Donsker-Varadhan information of $\nu$ with
respect to $\mu$ is defined by
\begin{equation}
I(\nu|\mu):=\begin{cases}\EE(\sqrt{f}, \sqrt{f}), \ \ &\text{ if }\ \nu=f\mu, \sqrt{f}\in\dd(\EE)\\
+\infty, &\text{ otherwise.}
\end{cases}
\end{equation}
\end{defn}

\noindent\textbf{Convention.}\  We adopt the following convention
for the Fisher-Donsker-Varadhan information on a Riemannian
manifold $\XX$: if $\nu=f \mu$ with $f>0$ smooth,
$$
I(\nu|\mu)= \frac 14 \int_\XX \frac{|\nabla f|^2}{f} \, d\mu.
$$
This means that $I=I_F/4$ where $I_F$ is the standard Fisher
information. This will lead to transportation-information
inequalities with a natural interpretation in terms of large
deviations, see (\ref{eq-05}) below.
\medskip

When $(P_t)$ is $\mu$-symmetric, $\nu\mapsto I(\nu|\mu)$ is
exactly the Donsker-Varadhan entropy i.e.\! the rate function
governing the large deviation principle of the empirical measure
$$
L_t:=\frac 1t\int_0^t \delta_{X_s} ds
$$
for large time $t$. This was proved by Donsker and Varadhan
\cite{DV75a, DV76, DV83} under some conditions of absolute
continuity and regularity of $P_t(x,dy)$, and established in full
generality by L.~Wu \cite[Corollary B.11]{Wu00b}. When $\mu=
e^{-V(x)}dx/Z$ ($Z$ is the normalization constant) with $V\in C^1$
on a complete connected Riemannian manifold $\XX=M$, the diffusion
$(X_t)$ generated by $\LL =\Delta -\nabla V\cdot \nabla$
($\Delta,\nabla$ are respectively the Laplacian  and the gradient
on $M$) is $\mu$-reversible and the corresponding Dirichlet form
is given by
$$
\EE_\mu(g,g)=\int_M |\nabla g|^2\, d\mu, \  g\in
\dd(\EE_\mu)=H^1(\XX, \mu)
$$
where $H^1(\XX, \mu)$ is the closure of $C_b^\infty(M)$ (the space
of infinitely differentiable functions $f$ on $M$ with
$|\nabla^nf|$ bounded for all $n$) with respect to the norm
$$
\sqrt{\mu(|g|^2 + |\nabla g|^2)}.
$$
It also matches with the space of these $g\in L^2(M)$ such that
$\nabla g\in L^2(M\to TM; \mu)$ in distribution. In this case, if
$\nu=f\mu$ with $0<f\in C^1(M)$, then
\begin{equation}\label{eq-05}
    I(\nu|\mu)=\int_\XX |\nabla \sqrt{f}|^2 \,d\mu= \frac 14
\int_\XX\frac{|\nabla f|^2}{f}\, d\mu.
\end{equation}

\subsubsection*{Feynman-Kac semigroup}
The derivation of the large deviation results for $L_t$ as $t$
tends to infinity is intimately related to the Feynman-Kac
semigroup
\begin{equation}\label{Feynman-Kac} P_t^u g(x):=\ee^x g(X_t)
\exp\left(\int_0^tu(X_s)\,ds\right).
\end{equation}
When $u$ is bounded, $(P_t^u)$ is a strongly continuous semigroup
of bounded operators on $L^2(\mu)$ whose generator is given by
$\LL^u g=\LL g +u g$, for all $g\in \dd_2(\LL^u)=\dd_2(\LL)$.
\\
It is no surprise that this semigroup also plays a role in the
present investigation.

\subsection{Characterizations of $T_cI$}

Recall that Kantorovich's duality theorem (see \cite{Vill03})
states that for any $\nu,\mu\in M_1(\XX)$ so that
$T_c(\nu,\mu)<+\infty$,
\begin{equation}\label{dual}
T_c(\nu,\mu)= \sup_{(u,v)\in \Phi_c} \int u\, d\nu - \int v\,d\mu
\end{equation}
where
$$
\Phi_c:=\{(u,v)\in (b\BB)^2: u(x)-v(y)\le c(x,y), \ \forall
(x,y)\in \XX^2\}.
$$
 This motivates us to introduce as in \cite{GL}

\begin{equation}\label{T_Phi}
T_\Phi(\nu,\mu)=\sup_{(u,v)\in\Phi} \int u\, d\nu - \int v\,d\mu
\end{equation}
where $\Phi\subset (b\BB)^2$ (non-empty) satisfies

\begin{itemize}
    \item[(A1)] $u\le v$ for all $(u,v)\in \Phi$ (or
equivalently $T_\Phi(\nu, \nu)=0$ for all $\nu\in M_1(\XX)$);
    \item[(A2)] For all $\nu_1,\nu_2\in M_1(\XX)$, there exists $(u,v)\in
\Phi$ such that $\int u\, d\nu_1 - \int v\,d\nu_2\ge 0$ (or
equivalently $T_\Phi(\nu_1, \nu_2)\ge 0$ for all $\nu_1, \nu_2\in
M_1(\XX)$).
\end{itemize}

Note that for (A1) and (A2) to be satisfied when $\Phi=\Phi_c$, it
is enough that $c(x,x)=0$ for all $x.$ The main result of this
section is the following generalization of Theorem \ref{thm-GL}.

\begin{thm}\label{thm-a2.1} Let $((X_t)_{t\ge0}, \pp_\mu)$ be a
stationary ergodic Markov process with the symmetrized Dirichlet
form $(\EE, \dd(\EE))$,  $\Phi$ be as above  and $\alpha:
[0,\infty)\to [0,\infty]$ be a left continuous increasing function
such that $\alpha(0)=0.$ Consider the following properties:

\begin{enumerate}
\item[(a)] The following transportation inequality holds
\begin{equation}
    \alpha(T_\Phi(\nu,\mu))\le I(\nu|\mu),\ \  \forall \nu\in M_1(\XX)
    \tag{$T_\Phi I$}
\end{equation}

\item[(b)] For all $(u,v)\in \Phi$ and all $\lambda, t\ge 0$

\begin{equation}\label{L2bound} \|P_t^{\lambda u}\|_{L^2(\mu)} \le
e^{t[\lambda \mu(v)+ \alpha^\circledast(\lambda)]}
\end{equation}
where $P_t^{\lambda u}$ is the Feynman-Kac semigroup
(\ref{Feynman-Kac}) and $ \alpha^\circledast$ is defined at
(\ref{b1}).

\item[$(b')$] For all $(u,v)\in \Phi$ and all $\lambda\ge 0$
\begin{equation*}
    \limsup_{t\rightarrow\infty}\frac 1t
    \log\mathbb{E}_\mu\exp\left(\lambda\int_0^tu(X_s)\,ds\right)\leq
    \lambda \mu(v)+ \alpha^\circledast(\lambda)
\end{equation*}

\item[(c)] For any initial measure $\beta\ll\mu$ with
$d\beta/d\mu\in L^2(\mu)$ and for all $(u,v)\in \Phi$ and $r,
t>0$,

\begin{equation}\label{thm-a2.1a} \pp_\beta\left(\frac 1t\int_0^t u(X_s)\,ds \ge
\mu(v)+ r\right)\le \left\|\frac{d\beta}{d\mu}\right\|_2 e^{-
t\alpha(r)}
\end{equation}

\item[$(c')$] For all $(u,v)\in \Phi$ and for any $r\ge0$, there
exists $\beta\in M_1(E)$ such that $\beta\ll\mu,$ $d\beta/d\mu\in
L^2(\mu)$ and
$$
\limsup_{t\to\infty}\frac 1t \log \pp_{\beta}\left(\frac
1t\int_0^t u(X_s)\,ds \ge \mu(v)+ r\right)\le -\alpha (r)
$$
\end{enumerate}
We have
\begin{enumerate}
    \item $(a)\Rightarrow (b)\Rightarrow (b')$ and
$(a)\Rightarrow (c)\Rightarrow (c')$.
    \item If $\alpha$ is convex, then $(a)\Leftrightarrow (b)$.
\item If $(P_t)$ is symmetric on $L^2(\mu)$, then
$(a)\Leftrightarrow (c) \Leftrightarrow (c')$.\\ If furthermore
$\alpha$ is convex, (a), (b), $(b')$, (c) and $(c')$ are
equivalent.
\end{enumerate}
\end{thm}

The proof of this result is similar to \cite[Theorems 2 and
15]{GL}'s ones. It takes advantage of large deviation results
previously obtained by L.~Wu. Namely,
\begin{enumerate}[-]
    \item the identification of the rate function in the symmetric case  and
the large deviation lower bound are taken from \cite{Wu00b} and
    \item the
non-asymptotic Cram\'er's upper bounds which are used in \cite{GL}
are replaced by the following result.
\end{enumerate}

\begin{lem}[L.~Wu \cite{Wu00c}]\label{lem21}
For any $u\in b\BB$ with $\mu(|u|)<+\infty$ and any $t>0$, the
following statements hold true.
\begin{enumerate}
    \item Denoting
    \begin{equation}\label{lem21a-2}
\Lambda(u):=\sup\left\{\int u g^2\, d\mu - \EE(g, g); g\in
\dd(\EE), \mu(g^2)=1, \mu(g^2|u|)<+\infty\right\},
 \end{equation}
 one has
\begin{equation}\label{lem21a}
\|P_t^u\|_{L^2(\mu)}\le e^{t\Lambda(u)}
\end{equation}
and  the equality holds in the symmetric case; \item For all
$r>0,$
 \begin{equation}\label{lem21b} \pp_\beta\left(\frac 1t\int_0^t
u(X_s)\,ds - \mu(u)\ge r\right)\le
\left\|\frac{d\beta}{d\mu}\right\|_2\exp\left(- t
\lim_{\delta\downarrow 0}I_u(\mu(u)+r-\delta)\right)
\end{equation}
where
$$
I_u(r):=\inf\left\{I(\nu|\mu);\nu\in M_1(\XX), \nu(u)=r\right\}, \
r\in \rr.
$$
\end{enumerate}
\end{lem}

It is proved in \cite{Wu00c, Wu00b} that in the symmetric case,
$I_u(r)$ is exactly the rate function governing the large
deviation principle of $\frac 1t\int_0^t u(X_s)\,ds$ for bounded
$u$. In these papers no mixing assumptions are required, this is
in contrast with the usual assumptions for the large deviation
principle as discovered by Donsker and Varadhan
\cite{DV75a,DV76,DV83} and reconsidered by Deuschel and Stroock
\cite{DS}. This relaxation of the usual assumptions is allowed by
the assumed restriction that the initial law is absolutely
continuous with respect to the ergodic measure $\mu.$

\bprf[Proof of Theorem \ref{thm-a2.1}] {\bf Part (1).} As $\nu\to
I(\nu|\mu)$ is convex on $M_1(\XX)$, so is $I_u:\rr\to
[0,+\infty]$. Since $I_u(\mu(u))=0$, $I_u$ is increasing on
$[\mu(u), +\infty)$. For all $(u,v)\in \Phi$ and all
$\lambda\ge0,$ we have
\begin{equation}\label{b2}
    \Lambda(\lambda u)=I_u^*(\lambda)
\end{equation}
where $I_u^*$ is the convex conjugate of $I_u.$ Indeed for
$\lambda\ge0$, by (\ref{lem21a-2})
$$
\aligned \Lambda(\lambda u)
&=\sup\{\lambda \int u g^2 \,d\mu- \EE(g, g); g\in \dd(\EE), \mu(g^2)=1\}\\
&= \sup\{\lambda \int u g^2 \,d\mu - \EE(g, g); 0\le g\in \dd(\EE), \mu(g^2)=1\}\\
&= \sup\{\lambda \int u \,d\nu - I(\nu|\mu); \nu\in M_1(\XX)\}\\
&=\sup_{a\in \rr} \{\lambda a - I_u(a)\}
\endaligned
$$
where the second equality follows from the fact that $\EE(|g|,
|g|)\le \EE(g,g)$ for all $g\in \dd(\EE).$

Note also  that $T_\Phi I$ implies that for any $(u,v)\in \Phi,$
\begin{equation}\label{b3}
I_u(\mu(v)+r)\ge \widetilde\alpha(r), \ \forall r\in\rr
\end{equation}
where $\widetilde \alpha(r)=\alpha(r)$ for $r\ge 0$ and $=0$ for
$r\le 0$. Indeed it is trivial for $r\le 0$ and for any $r\ge0$
and $\nu\in M_1(\XX)$ such that $\nu(u)=\mu(v)+r$,  $T_\Phi I$
implies that
$$
I(\nu|\mu)\ge \alpha(T_\Phi(\nu,\mu))\ge
\alpha(\nu(u)-\mu(v))=\alpha(r).
$$

\noindent$\bullet$\ $(a)\Rightarrow (b)$: Putting together
(\ref{b2}) and (\ref{b3}) leads us to
$$
\Lambda(\lambda u)= \sup_{a\in\rr} \left[\lambda a -
I_u(a)\right]\le \sup_{r\in\rr}\left[\lambda (\mu(v)+r) -
\widetilde\alpha(r)\right]\} =\lambda \mu(v) +
\alpha^\circledast(\lambda)
$$
for all $\lambda\ge0.$ Statement (b) now follows from inequality
(\ref{lem21a}).

\vskip10pt \noindent$\bullet$\ $(a)\Rightarrow (c)$: This follows
from (\ref{lem21b}) and (\ref{b3}), noting that by (A1),
$\mu(u)\le\mu(v)$ for all $(u,v)\in\Phi.$

\vskip10pt \noindent$\bullet$\ $(b)\Rightarrow (b')$ and
$(c)\Rightarrow (c')$: These implications are trivial.

\vskip20pt \noindent{\bf Part (2).} $(b)\Rightarrow (a)$ in the
case where $\alpha$ is convex. By (\ref{L2bound}), we have for
$(u,v)\in\Phi$ fixed and for any $g\in \dd_2(\LL)$,
$$
\<P_t^{\lambda u}g, P_t^{\lambda u}g\>_\mu\le e^{2t(\lambda \mu(v)
+ \alpha^\circledast(\lambda))} \<g,g\>_\mu.
$$
Differentiating at time zero we obtain
$$ 2\<g, \LL g + \lambda ug\>_\mu=
2(\lambda \mu(g^2u)-\EE(g,g))\le 2 (\lambda \mu(v) +
\alpha^\circledast(\lambda))\mu(g^2).
$$
Then for all $g\in\dd_2(\LL)$,
$$
\lambda[\mu(g^2 u) - \mu(v)\mu(g^2)]- \alpha^\circledast(\lambda)
\mu(g^2) \le  \EE(g,g).
$$
It can be extended to $g\in\DD(\EE)$. Now for any $\nu\in
M_1(\XX)$ such that $I(\nu|\mu)<+\infty$, applying the above
inequality to $g=\sqrt{\frac{d\nu}{d\mu}}$, we get
$$
\lambda[\nu(u)-\mu(v)]-\alpha^\circledast(\lambda)\le I(\nu|\mu).
$$
Taking the supremum over all $\lambda\in\mathbb{R}$, as $\alpha$
assumed to be convex and $\alpha^\circledast=\tilde{\alpha}^*$ on
$[0,\infty)$ (see the remark below (\ref{b1})), we get
$$
\tilde{\alpha}(\nu(u)-\mu(v)) \le I(\nu|\mu)
$$
and taking the supremum over all $(u,v)\in\Phi$ leads to the
desired result.

\vskip20pt \noindent{\bf Part (3).}  Let us assume from now on
that the semigroup $(P_t)$ is symmetric in $L^2(\mu).$

\par \noindent$\bullet$ $(c')\Rightarrow
(a):$ By the large deviation lower bound in \cite[Theorem
B.1]{Wu00b} and the identification of the rate function in the
symmetric case in \cite[Corollary B.11]{Wu00b}, we have for any
initial probability measure $\beta\ll\mu$,
$$
\liminf_{t\to\infty}\frac 1t \log \pp_{\beta}\left(\frac
1t\int_0^t u(X_s)\,ds \ge \mu(v)+ r\right)\ge -\inf\{I(\nu|\mu);\
\nu(u)>\mu(v)+r\}.
$$
This together with $(c')$ implies that for any $r\ge0$,
$$
\inf\{I(\nu|\mu);\ \nu(u)>\mu(v)+r\}\ge \alpha(r).
$$
Fix now $\nu$ such that $r_0=T_\Phi(\nu,\mu)>0$ (otherwise $T_\Phi
I$ is obviously true.) Choosing a sequence $(u_n,v_n)\in \Phi$ so
that $\nu(u_n)-\mu(v_n)> r_0-1/n,$ for all large enough  $n,$
$$
\alpha(r_0 -1/n) \le I(\nu|\mu)
$$
where $T_\Phi I$ follows by letting $n\to\infty$ and by the
left-continuity of $\alpha$.

\vskip10pt
 \noindent$\bullet$ $\alpha$ is convex and $(P_t)$ is
symmetric. $(b')\Rightarrow (c')$ with $\beta=\mu$: The proof is
standard and consists in optimizing exponential upper bounds. So
doing, one obtains by means of $(b')$ the asymptotic upper bound
$(c')$ with the convex envelope of $\tilde{\alpha}$ instead of
$\tilde{\alpha}.$ As $\alpha$ is assumed to be convex, $(c')$ is
proved.
\\
This completes the proof of  the theorem. \nprf

We now investigate two particular cases of Theorem \ref{thm-a2.1}.

\bcor[The inequalities $W_1I(c)$ and $W_2I(c)$]\label{cor21} Let
$c>0$ and let $(X_t)$ be a $\mu$-reversible and ergodic Markov
process such that $\int d^2(x,x_0)\,d\mu(x)<+\infty.$

\begin{enumerate}[(1)]
\item  The statements below are equivalent:

\begin{enumerate}

    \item The following $W_1I(c)$ inequality  holds true:
\begin{equation}
     W_1^2(\nu, \mu) \le 4c^2\, I(\nu|\mu),\ \forall \nu\in
    M_1(\XX);
    \tag{$W_1I(c)$}
\end{equation}

    \item For all Lipschitz function $u$ on $\XX$ with
$\|u\|_{\mathrm{Lip}}\le 1$ and all $\lambda, t\ge0,$
$$\|P_t^{\lambda u}\|_{L^2(\mu)} \le \exp\left(
\lambda \mu(u)+ c^2\lambda^2\right);$$

    \item For all Lipschitz function $u$ on $\XX$ with
$\|u\|_{\mathrm{Lip}}\le 1, \ \mu(u)=0$ and all $\lambda\ge 0$,
$$
\limsup_{t\to +\infty} \frac 1t \log \ee_\mu \exp\left(\lambda
\int_0^t u(X_s)\,ds\right)\le c^2\lambda^2;
$$

    \item For all Lipschitz function $u$ on $\XX,$  $r> 0$ and $\beta\in
    M_1(\XX)$such that $d\beta/d\mu\in L^2(\mu),$
$$
\pp_\beta\left(\frac 1t\int_0^t u(X_s)\,ds\ge\mu(u)+r\right)\le
\left\|\frac{d\beta}{d\mu}\right\|_2 \exp\left(-\frac{
r^2}{4c^2\|u\|_{\mathrm{Lip}}^2}\right).
$$
 \end{enumerate}

 \item The statements below are equivalent:

\begin{enumerate}[(a)]

\item The following $W_2I(c)$ inequality holds true:
\begin{equation}
     W_2^2(\nu, \mu) \le 4c^2 I(\nu|\mu),\ \forall \nu\in
    M_1(\XX);
    \tag{$W_2I(c)$}
\end{equation}
\item For any $v\in b\BB$,

$$\|P_t^{\frac{1} {4c^2} Qv}\|_{L^2(\mu)} \le e^{\frac{t} {4c^2} \mu(v)}, \ \forall t\ge0$$
where  $\displaystyle{ Qv(x)=\inf_{y\in \XX} \{v(y) + d^2(x,y)\}
}$ is the so-called ``inf-convolution'' of $v$;

 \item For any $u\in b\BB$,
$$\|P_t^{\frac{1} {4c^2} u}\|_{L^2(\mu)} \le e^{\frac{t} {4c^2} \mu(Su)},\ \forall t\ge0$$
where  $\displaystyle{ Su(y)=\sup_{x\in \XX} \{u(y) - d^2(x,y)\}
}$ is the so-called ``sup-convolution'' of $u$.
 \end{enumerate}
 \end{enumerate}

\ncor

\noindent{\bf Notation.}\ The best constants $c>0$ in $W_1I(c)$
and $W_2I(c)$ will be denoted respectively by $\WI(\mu)$ and
$\WWI(\mu)$.

\brmk{\rm\
\begin{enumerate}[(i)]
    \item The  best constants $\WI(\mu)$ and
$\WWI(\mu)$ depend on the metric $d$ and the Dirichlet form $\EE$.
Of course $\WI(\mu)\le \WWI(\mu)$.
    \item The  above corollary may be seen as the counterpart of
Bobkov-G\"otze's characterizations of $W_pH$ ($p=1,2$) for Markov
processes.
    \item For a justification of the choice of the constant $4c^2$ in
$WI(c),$ see Example \ref{ex-01} below, where $\WI$ and $\WWI$ are
identified as standard error in a Gaussian model.
\end{enumerate}
}\nrmk

\bprf {\bf Part (1).}  Notice that
$W_1(\nu,\mu)=T_{\Phi}(\nu,\mu)$, where $\Phi:=\{(u,u);
\|u\|_{\mathrm{Lip}}\le 1, u\in b\BB\}$. The result is a direct
consequence of Theorem \ref{thm-a2.1} in case $u$ is bounded. The
transition from a bounded to an unbounded $u$ follows from an
elementary monotone convergence argument.

\noindent{\bf Part (2).} The equivalence of (b) and (c) is direct.
Though $(a)\Leftrightarrow (b)$ follows easily from Theorem
\ref{thm-a2.1}, we nevertheless present  a simple proof. In the
present symmetric case, by Lemma \ref{lem21}, for any $u\in
L^1(\mu)$ the equality is achieved in inequality (\ref{lem21a}):
$$
\|P_t^u\|_{L^2(\mu)} = e^{t\Lambda(u)}, \ \forall t\ge0
$$
(possibly infinite) where $\Lambda(u)$ is given at
(\ref{lem21a-2}). Notice that $Qv$ is upper semicontinuous and
$Su$ is lower semicontinuous.

\noindent$\bullet$\ $(a)\Rightarrow(b)$. Since for any $\nu\in
M_1(E)$,
$$
\int Qv\, d\nu - \int v\,d\mu \le W_2^2(\nu,\mu) \le
4c^2\,I(\nu|\mu)
$$
then
$$ \frac{1} {4c^2} \int Qv \,d\nu - I(\nu|\mu) \le \frac{1} {4c^2} \mu(v)
$$
Taking the supremum over all $\nu$ yields (b).

\noindent$\bullet$\ $(b)\Rightarrow(a)$. Reverse the above proof.
 \nprf

\brmk \rm
 We have seen that, by Theorem \ref{thm-a2.1}, $T_\Phi I$ inequalities lead to
 exponential deviation inequalities when the initial measure $\beta$ is such that
 $d\beta/d\mu \in L^2(\mu)$. It is of course a limitation for the applications.
 Let us see that in the diffusion case we may overcome this limitation.  As
 remarked by Wu \cite[p.441-442]{Wu00c} , this assumption can be replaced
 by $d\beta/d\mu \in L^q(\mu)$ for $1\leq q < 2$, provided that one
 replaces
 (\ref{lem21a})
 in Lemma \ref{lem21} by $\|P_t^u\|_p\le e^{t\Lambda_p(u)}$, with
$$\Lambda_p(u) := \sup \, \left\{ \int \, u|f|^p \, d\mu \, + \, \langle\textrm{sgn}(f)|f|^{p-1},\LL f\rangle_\mu \, ; \, f\in \dd_p(\LL) \,
\textrm{ and } \, \int |f|^p d\mu =1 \right\} \, ,$$
 where $p$ and $q$ are conjugate numbers. Now, suppose that
$\LL$ admits a carr\'e du champ $\Gamma.$ One can integrate by
parts and get
$$\langle\textrm{sgn}(f)|f|^{p-1},\LL f\rangle_\mu \, =
\, - \, (4(p-1)/p^2) \, \int \Gamma(|f|^{p/2}) \, d\mu \, .$$
Taking $g=|f|^{p/2}$ in the definition of $\Lambda,$ we obtain
that
$$\Lambda_p(u) \, = \, (4(p-1)/p^2) \, \Lambda( (p^2/4(p-1)) u).$$ Once again a deviation inequality is obtained, however with worse constants.
\nrmk

 %%%%
 %%%%

\subsection{Tensorization of $T_cI$}

Assume that $\mu_i\in M_1(\XX_i)$ satisfies

\begin{equation}\label{a21} \alpha_i(T_{c_i}(\nu, \mu_i))\le I_i(\nu|\mu_i),\
\forall \nu\in M_1(\XX_i) \end{equation} where $I_i(\nu|\mu_i)$ is
the Fisher-Donsker-Varadhan information related to the Dirichlet
form $(\EE_i, \dd(\EE_i))$, and $\alpha_i$ is moreover convex. On
the product space $\XX^{(n)}:=\prod_{i=1}^n \XX_i$ equipped with
the product measure $\mu:=\prod_{i=1}^n \mu_i$, consider the
sum-cost function

\begin{equation}\label{a22}
\oplus_ic_i (x,y) :=\sum_{i=1}^n c(x_i, y_i), \ \forall x,y\in
\XX^{(n)}
\end{equation}
and the {\it inf-convolution} of $(\alpha_i)$
\begin{equation}\label{a23}
\alpha_1\Box\cdots\Box
\alpha_n(r):=\inf\left\{\sum_{i=1}^n\alpha(r_i);\ r_i\ge 0,
\sum_{i=1}^nr_i=r\right\}.
\end{equation}
It also shares the following properties of every $\alpha_i:$ it is
increasing, left continuous and convex on $\rr^+$ with
$\alpha(0)=0$ (see \cite{GL}). Define the sum-Dirichlet form of
$\oplus_i\EE_i$ by
\begin{equation}\label{a24}
\aligned \dd(\oplus_i\EE_i)&:=\left\{g\in L^2(\mu): g_i\in
\dd(\EE_i), \textrm{for }
\mu\textrm{-a.e. } \hat x_i \textrm{ and } \int_{\XX^{(n)}} \sum_{i=1}^n\EE_i(g_i,g_i)\,d\mu<+\infty\right\}\\
\oplus_i\EE_i(g, g)&:= \int_{\XX^{(n)}} \sum_{i=1}^n
\EE_i(g_i,g_i)\,d\mu, \quad  g\in \dd(\EE)
\endaligned
\end{equation}
where $g_i(x_i):=g(x_1,\cdots, x_i, \cdots, x_n)$ with $\hat
x_i:=(x_1,\cdots, x_{i-1}, x_{i+1}, \cdots, x_n)$ fixed.

\begin{thm}\label{thm22}
Assume (\ref{a21}) for each $i=1,\cdots, n$ with $\alpha_i$
moreover convex.  Define $c, \alpha, \EE$ respectively by
(\ref{a22}), (\ref{a23}) and (\ref{a24}). Let
$I_{\oplus_i\EE_i}(\nu|\mu)$ be the Fisher-Donsker-Varadhan
information associated with $(\oplus_i\EE_i, \dd(\oplus_i\EE_i))$.
Then
\begin{equation}
\alpha_1\Box\cdots\Box \alpha_n(r) (T_{\oplus c_i}(\nu, \mu))\le
I_{\oplus_i\EE_i}(\nu|\mu),\ \forall \nu\in
M_1\left(\XX^{(n)}\right).
\end{equation}
\end{thm}

This result is similar to \cite[Corollary 5]{GL}, but the proof
will be different. It is based on the following sub-additivity
result for the transportation cost of a product measure which is
different from Marton's original proof \cite{Mar96} where an
ordering of sites is required.

\begin{lem}\label{lem23} Given a probability measure $\nu$ on
$\prod_{i=1}^n \XX_i$, let $\nu_i$ be the regular conditional
distribution of $x_i$ knowing $\hat x_i$. Then with the cost
function $c$  given at (\ref{a22}),
$$
T_{\oplus c_i}(\mu,\nu) \le \ee^\nu \sum_{i=1}^n T_{c_i}(\mu_i,
\nu_i).
$$
\end{lem}

\noindent \textbf{Notation.}\ The expectation $\ee^\nu$ simply
means integration with respect to $\nu.$

\bprf Let $(Z_i=(X_i, Y_i))_{i=1,\cdots, n}$ be a sequence of
random variables  valued in $\prod_{i=1}^n \XX_i^2$ defined on
some probability space $(\Omega, \FF, \pp)$,  realizing $T_{\oplus
c_i}(\mu,\nu)$, i.e., the law of $X=(X_i)_{i=1,\cdots, n}$ is
$\mu=\prod_{i=1}^n\mu_i$, the law of $Y=(Y_i)_{i=1,\cdots, n}$ is
$\nu$, and
$$
 \ee \sum_{i} c_i(X_i, Y_i)= T_{\oplus c_i}(\mu,\nu).
$$
For each $i$ fixed, construct a couple of r.v. $(\tilde X_i,
\tilde Y_i)$ so that its conditional law given $(Z_j)_{j\ne i}$ is
a coupling of $(\mu_i(dx_i), \nu_i(dx_i| Y_j, j\ne i)$ and
$\pp$-a.s.,
$$
\ee [c_i(\tilde X_i, \tilde Y_i)|Z_j, j\ne i]= T_{c_i}(\mu_i,
\nu_i(\cdot| Y_j, j\ne i)).
$$
Obviously $(X_j, j\ne i; \tilde X_i)$ and  $(Y_j, j\ne i; \tilde
Y_i)$ (more precisely their joint law) constitute  a coupling of
$(\mu,\nu)$. Thus $\ee \sum_{j} c_j(X_j, Y_j)\le \ee [\sum_{j\ne
i} c_j(X_j, Y_j) + c_i(\tilde X_i, \tilde Y_i)]$ or
$$\ee c_i(X_i,
Y_i)\le \ee c_i(\tilde X_i, \tilde Y_i)= \ee T_{c_i}(\mu_i,
\nu_i(\cdot| Y_j, j\ne i)).$$ Consequently
$$
T_{\oplus c_i}(\mu,\nu)= \ee \sum_{i=1}^nT_{c_i}(X_i, Y_i) \le \ee
\sum_{i=1}^n T_{c_i}(\mu_i,\nu_i(\cdot| Y_j, j\ne i))= \int
\sum_{i=1}^n T_{c_i}(\mu_i, \nu_i)\,d\nu.
$$
\nprf

The following additivity property of the Fisher information will
be needed. It holds even in the dependent case.

\begin{lem}\label{lem26} Let $\nu, \mu$ be probability measures on
$\prod_{i=1}^n \XX_i$ such that $I(\nu|\mu)<+\infty$, let $\mu_i$,
$\nu_i$  be the regular conditional distributions of $x_i$ knowing
$\hat x_i$ under $\mu,\ \nu$. Then \begin{equation}\label{lem26a}
I_{\oplus_i\EE_i}(\nu|\mu) = \ee^\nu \sum_i I_i(\nu_i|\mu_i).
\end{equation} \end{lem}

\bprf Let $f=d\nu/d\mu$. Then $d\nu_i/d\mu_i= f/ \mu_i(f)=
f_i/\mu_i(f_i), \nu$-a.s. (recalling that $f_i$ is the function
$f$ of $x_i$ with $\hat x_i$ fixed). For $\nu$-a.e. $\hat x_i$
fixed,
$$
I_i(\nu_i|\mu_i)= \EE_i\left(\sqrt{\frac{f_i}{\mu_i(f_i)}},
\sqrt{\frac{f_i}{\mu_i(f_i)}}\right)= \frac 1{\mu_i(f_i)}
\EE_i(\sqrt{f_i}, \sqrt{f_i})
$$
(for $\mu_i(f_i)$ is constant with $\hat x_i$ fixed). We obtain
$$
\aligned \ee^{\nu}\sum_{i=1}^n I_i(\nu_i|\mu_i) &= \ee^{\mu}f
\sum_{i=1}^n
\frac 1{\mu_i(f_i)} \EE_i(\sqrt{f_i}, \sqrt{f_i})\\
&=
\ee^{\mu} \sum_{i=1}^n  \EE_i(\sqrt{f_i}, \sqrt{f_i})\\
&=\oplus_i\EE_i (\sqrt{f}, \sqrt{f})=I_{\oplus_i\EE_i}(\nu|\mu),
\endaligned
$$
which completes the proof. \nprf

The above additivity  is different from the super-additivity of
the Fisher information for product measure obtained by E. Carlen
\cite{Car91}.

\bprf[Proof of Theorem \ref{thm22}] Without loss of generality we
may assume that $I(\nu|\mu)<+\infty$. For simplicity write
$\alpha=\alpha_1\Box\cdots\Box \alpha_n$. By Lemma \ref{lem23},
Jensen's inequality  and the definition of $\alpha$,
$$
\aligned \alpha(T_{\oplus c_i}(\nu,\mu))
&\le \alpha\left(\ee^{\nu}\sum_{i=1}^n T_{c_i}(\nu_i, \mu_i)\right)\\
&\le \ee^{\nu}\alpha\left(\sum_{i=1}^n T_{c_i}(\nu_i, \mu_i)\right)\\
&\le \ee^{\nu}\sum_{i=1}^n\alpha_i( T_{c_i}(\nu_i, \mu_i))\\
&\le \ee^{\nu}\sum_{i=1}^n I_i(\nu_i|\mu_i).
\endaligned
$$
The last quantity is equal to $I_{\oplus\EE_i}(\nu|\mu)$, by Lemma
\ref{lem26}. \nprf

As an example of application, let $(X^i_t)_{t\ge 0}, i=1,\cdots,
n$ be $n$ Markov processes with the same transition semigroup
$(P_t)$ and the same symmetrized Dirichlet form $\EE$ on
$L^2(\mu)$, and conditionally independent once
$(X^i_0)_{i=1,\cdots, n}$ is fixed. Then $X_t:=(X^1_t,\cdots,
X^n_t)$ is a Markov process with the symmetrized Dirichlet form
given by
$$
\oplus_n\EE(g, g)= \int \sum_{i=1}^n \EE(g_i, g_i)\,
\mu(dx_1)\cdots\mu(dx_n)
$$
which is the $n$-fold sum-Dirichlet form of $\EE$.

\bcor\label{cor24} Assume that $\mu$ satisfies $T_cI$ on $\XX$
with $\alpha$ convex.  Then $\mu^{\otimes n}$ satisfies

\begin{equation}\label{cor24a}
n \alpha\left(\frac{T_{\oplus_n c}(\nu, \mu^{\otimes n})}{n}
\right) \le I_{\oplus_n\EE}(\nu|\mu^{\otimes n}),\ \forall \nu\in
M_1(\XX^n).
\end{equation}
In particular for all $(u,v)\in \Phi_c$, for all initial measure
$\beta$ on $\XX^n$ with $d\beta/d\mu^{\otimes n}\in
L^2(\mu^{\otimes n})$ and for any $t,r>0$,
\begin{equation}\label{cor24b}
\pp_\beta\left(\frac 1n\sum_{i=1}^n \frac 1t \int_0^t u(X^i_s)\,ds
\ge\mu(v)+r\right)\le \left\|\frac{d\beta}{d\mu^{\otimes
n}}\right\|_2 e^{-nt\alpha (r)}.
\end{equation}
\ncor

\bprf As $\alpha^{\Box n}(r)= n \alpha(r/n),$ (\ref{cor24a})
follows from Theorem \ref{thm22}. Noting that for $u,v\in \Phi_c$,
$(\sum_{i=1}^nu(x_i), \sum_{i=1}^nv(x_i))$ as a couple of
functions on $\XX^n$ belongs to $\Phi_{\oplus_n c}$, we obtain
(\ref{cor24b}) by Theorem \ref{thm-a2.1}. \nprf

The tensorization of $W_pI$ in the dependent Gibbs measure case is
carried out in Gao and Wu \cite{p-GaoW}.

\subsection{Relations between $W_2I,$  Poincar\'e and log-Sobolev inequalities}

In the rest of the paper we are interested in two particular cases
of $T_cI$: $W_1I(\kappa)$ and $W_2I(\kappa)$ introduced at
Corollary \ref{cor21}.
\\
\noindent \textbf{Notation} (Spectral gap). As usual, one says
that $\mu$ satisfies a Poincar\'e inequality if
\begin{equation*}
    \Var_\mu(g) \le c\,\EE(g, g), \ \forall g\in
\dd_2(\LL)
\end{equation*}
for some \emph{finite} $c\ge0$ and a  Dirichlet form $\EE$ which
is closable in $L^2(\mu).$ We denote $\SG(\mu)$ the best constant
$c$ in the above Poincar\'e inequality. It is the inverse of the
\emph{spectral gap} of $\LL.$

From the work of Otto-Villani \cite{OVill00}, we have the
following  observations.

\begin{prop}\label{prop25} Let $\XX$ be a complete connected Riemannian
manifold and $\mu= e^{-V(x)} dx/Z$ where  $dx$ is the Riemannian
volume measure, $V\in C^2(\XX)$ and $Z=\int_\XX e^{-V}
dx<+\infty$. Let $\dd(\EE)$ be the space $H^1(\XX,\mu)$ of those
functions $g\in L^2(\XX,\mu)$ such that $\nabla g\in L^2(TM, \mu)$
in the sense of distribution and consider the Dirichlet form,
$$
\EE_\nabla (g,g):= \int_\XX |\nabla g|^2 \,d\mu, \ g\in \dd(\EE)
$$
and the associated Fisher-Donsker-Varadhan information
$I(\nu|\mu),$ see (\ref{eq-05}).
\begin{enumerate}[(a)]
\item If the log-Sobolev inequality below
$$
 H(\nu|\mu) \le 2c\, I(\nu|\mu),\ \forall \nu
$$
is satisfied, then $\mu$ satisfies $W_2I(c)$. In other words the
best constant $\LS(\mu)$ in the log-Sobolev  inequality above
satisfies

$$\LS(\mu)\ge \WWI(\mu).$$

 \item If $W_2I(c)$ holds, then the Poincar\'e
inequality holds with constant $c$. In other words the inverse
spectral gap $\SG(\mu)$ of $\EE_\nabla$ satisfies $$\WWI(\mu)\ge
\SG(\mu).$$

 \item Assume that the Bakry-Emery curvature $\rm{Ric} + {\rm Hess} V$ is
bounded from below by $K\in\rr$, where  $\rm{Ric}$ is the Ricci
curvature and ${\rm Hess} V$ is the Hessian of $V.$ If $W_2I(c)$
holds with $cK\le 1$ (this is possible by Part (a) and
Bakry-Emery's criterion), then the log-Sobolev inequality
$$
H(\nu|\mu) \le 2(2c - c^2K)\,I(\nu|\mu),\ \forall \nu
$$
is also satisfied or equivalently
$$
2\WWI(\mu) - K\WWI(\mu)^2\ge \LS(\mu) .
$$
\end{enumerate}
\end{prop}

 \bprf Before the proof, let us remind the reader that $I=I_F/4$ where
$I_F$ is $I$ in Otto-Villani's paper \cite{OVill00}.
\\\medskip\noindent $\bullet$ Statement (a). The
proof is direct, as by \cite{OVill00} or \cite{BGL01} a
logarithmic Sobolev inequality implies the $W_2H$ (sometimes
called $T_2$) inequality so that
$$W_2(\nu,\mu)\le\sqrt{ 2c H(\nu|\mu)}\le 2c \sqrt{I(\nu|\mu)}$$
which is the announced conclusion.\\\medskip

\noindent$\bullet$ Statement (b). The proof follows from the usual
linearization procedure. Set $\mu_\vep=(1+\vep g)\mu$ for some
smooth and bounded $g$ with $\int g\,d\mu=0$, we easily get: as
$\vep\to 0$,
$$I(\mu_\vep|\mu)/\vep^2\to \frac 14 \EE_\nabla(g,g)$$
and by Otto-Villani \cite[p.394]{OVill00}, there exists $r$ such
that
$$
\int g^2\,d\mu\le
\sqrt{\EE_\nabla(g,g)}\frac{W_2(\mu_\vep,\mu)}{\vep}
+\frac{r}{\vep}W_2^2(\mu_\vep,\mu).
$$
Using now $W_2I(c)$ we get
$$
\int g^2\,d\mu\le
2c\sqrt{\EE_\nabla(g,g)}\sqrt{\frac{I(\mu_\vep|\mu)}{\vep^2}}
+\frac{4rc^2}{\vep}I(\mu_\vep|\mu).
$$
Letting $\vep\to0$ gives the result.\\\medskip

\noindent$\bullet$ Statement (c). This result is a direct
application of the HWI inequality \cite[Th.3]{OVill00} in the
Euclidean case and \cite{BGL01} for a general Riemannian manifold:

\begin{equation}\label{HWI-OV}
H(\nu|\mu)\le 2 W_2(\nu,\mu)\sqrt{I(\nu|\mu)}- \frac K2 W_2^2(\nu,\mu).
\end{equation}
\nprf

\bexa\label{ex-01} \rm Let $\mu=\NN(0,\Sigma)$ be the centered
Gaussian measure on $\rr^n$, with positive definite covariance
matrix $\Sigma$. We claim that with respect to the Euclidean
metric $|\cdot|$,
$$
\WI(\mu)= \WWI(\mu)=\lambda^{-1}_{\mathrm{min}}(\Sigma^{-1})
$$
where $\lambda_{\mathrm{min}}$ denotes the minimal eigenvalue.
Indeed it is well known that with respect to the usual gradient
$\nabla$ on $\rr^n$ (cf. Ledoux \cite{Led01}),
$$
\SG(\mu)=\LS(\mu)=\lambda^{-1}_{\mathrm{min}}(\Sigma^{-1}).
$$
Then by Proposition \ref{prop25}(a) and (b),
$$
\WI(\mu)\le \WWI(\mu)=\lambda^{-1}_{\mathrm{min}}(\Sigma^{-1}).
$$
On the other hand, let $\nu=\NN(m,\Sigma)$ where $m\in\rr^n$. Then
$W_1(\nu,\mu)=|m|$ (indeed $W_p(\nu,\mu)=|m|$ for all $p\ge1$),
and $I(\nu|\mu)=\frac 14 |\Sigma^{-1}m|^2$. Thus
$$
\WI(\mu) \ge \left(\inf_{m\in\rr^n}
\frac{|\Sigma^{-1}m|}{|m|}\right)^{-1}=\lambda^{-1}_{\mathrm{min}}(\Sigma^{-1})
$$
completing the proof of the claim. \nexa

As is seen from this proposition, in the bounded below curvature
case, $W_2I(\kappa)$ is (qualitatively as constants are lost)
equivalent to a logarithmic Sobolev inequality. It is an
interesting question to know whether or not it is the case in full
generality. We do not believe this to be true, and the hint for
this conjecture comes from the fact that even for the usual $W_2H$
inequality, the only cases of measures satisfying $W_2H$ but not
log-Sobolev inequality which are known so far have an infinite
curvature \cite{p-CG04}. That is why in the rest of the paper we
are mainly interested in $W_1I(\kappa).$ Nevertheless, note that
$W_2I$ may easily be applied to obtain tensorization results in
some situations where dependence occurs and the log-Sobolev
inequality seems to be unfruitful, see \cite{p-GaoW}.

\section{Poincar\'e inequality implies Hoeffding's deviation inequality}

\subsection{Relations between Poincar\'e  and $W_1I$ inequalities}
The purpose of this section is to establish

\begin{thm}\label{thm2}
Let $((X_t), \pp_\mu)$ be a stationary  ergodic Markov process.
\begin{enumerate}[(a)]

\item The Poincar\'e inequality
\begin{equation}\label{Poincare}
 \Var_\mu(g) \le \SG\,\EE(g, g), \ \forall g\in
\dd_2(\LL)
\end{equation}
implies
\begin{equation}\label{thm2b}
\|\nu -\mu\|_{\textrm{TV}}^2\le 4\SG\, I(\nu|\mu),\ \forall \nu\in
M_1(\XX).
\end{equation}
In particular for every initial probability measure $\beta\ll\mu$
with $d\beta/d\mu\in L^2(\mu)$ and for all $u\in b\BB$, $t,r>0$,
\begin{equation}\label{thm2a} \pp_\beta\left(\frac 1t\int_0^t u(X_s)\,ds -
\mu(u)\ge r\right)\le
\left\|\frac{d\beta}{d\mu}\right\|_2\exp\left(-\frac{t
r^2}{\SG\delta(u)^2}\right)
\end{equation}
where  $\delta(u):=\sup_{x,y\in \XX}|u(x)-u(y)|$ is the
oscillation of $u.$

\item Conversely, under the additional assumption that $(X_t,
\pp_\mu)$ is reversible and $R_1:=\int_0^\infty e^{-t}P_t\, dt$ is
$\mu$-uniformly positive improving, if there is some
left-continuous and increasing $\alpha$ such that $\alpha(r)>0$
for all $r>0$ and
$$
\alpha\left(\|\nu -\mu\|_{\textrm{TV}}\right)\le I(\nu|\mu), \
\forall \nu\in M_1(\XX),
$$
then the Poincar\'e inequality (\ref{Poincare}) holds.

\item In other words when $d$ is the trivial metric,  $W_1I$ is
equivalent to  Poincar\'e's inequality in the symmetric and
uniformly positive improving case.
\end{enumerate}
\end{thm}

Here the kernel $R_1$ is said $\mu$-uniformly positive improving,
if for any $\vep>0$,
$$
\inf_{A,B: \mu(A), \mu(B)\ge \vep} \mu(1_A R_11_B)>0.
$$
Note that if the symmetric semigroup $(P_t)$ is irreducible,
i.e.\! $R_1(x,A)=\int_0^\infty e^{-t} P_t(x,A)\,dt>0, \ \forall
x\in \XX$ for any $A\in\BB$ charged by $\mu$, then $R_1$ is
$\mu$-uniformly positive improving for every $t>0$ (cf.
\cite{p-Wu01, GW06}). A typical example of this situation is when
$P_t(x,dy)=p_t(x,y) \mu(dy)$ and $p_t(x,y)>0$,
$\mu\otimes\mu$-a.s.\! for some $t>0$. See \cite{p-Wu01, GW06} for
related results and references on this subject.

\brmk\label{rem31}{\rm \
\begin{enumerate}[(i)]
    \item Let $d(x,y)= \mathbf{1}_{x\ne y}$ (the trivial metric)
    and $\Phi=\{(u,u); \delta(u)\le 1\}$. Then
$$
\frac 12 \|\nu-\mu\|_{\textrm{TV}} =W_1(\nu,\mu)=T_\Phi(\nu,\mu).
$$
Hence (\ref{thm2b}) is exactly the inequality $T_\Phi I$ or
$W_1I(c)$ with $4c^2=\SG,$ and (\ref{thm2a}) is a direct
consequence of (\ref{thm2b}) and Corollary \ref{cor21}-(1-d).

    \item Hoeffding type inequality (\ref{thm2a}) improves a
similar result by Cattiaux and Guillin \cite{p-CG06}. Lezaud
\cite{Lez01} proved a deviation inequality, which is better than
(\ref{thm2a}) in the moderate deviation scale ($r$ very small),
nevertheless his proof involves a difficult argument based on
Kato's theory of perturbation of operators.

    \item Inequality (\ref{thm2a}) is only meaningful for $r$ small
    enough, since its left-hand side
    vanishes as soon as $r>\delta(u).$

    \item About the deviation inequality in Theorem
\ref{thm-a2.1}(c), several variants are already known. K.~Marton
\cite{Mar97} proved a Gaussian deviation inequality for Doeblin
recurrent Markov chains $(X_n)_{n\in \nn}$ by means of
$L^1$-transportation inequality for the law of the chain. Her
result is next generalized in Djellout and al. \cite{DGW03}.

\item Does $W_1I(c)$ imply the exponential decay: $W_1(\nu P_t,
\mu)\le C e^{-\delta t} W_{1}(\nu,\mu) ?$ In the trivial metric
case this decay means Doeblin recurrence and Theorem \ref{thm2}
gives a negative answer. Indeed, there exist reversible Markov
processes having a positive spectral gap which are not Doeblin
recurrent, for instance the Ornstein-Uhlenbeck processes.

\end{enumerate}
 }\nrmk

\bprf[Proof of Theorem \ref{thm2}] $\bullet$ Statement (a). As
noticed at Remark \ref{rem31}-(i), all we have to prove is the
transportation inequality (\ref{thm2b}). To this end we may assume
that $f=d\nu/d\mu$ satisfies $\sqrt{f}\in \dd(\EE)$. By the
inequality (\ref{thm1a}) in Theorem \ref{thm3} below and the
assumed Poincar\'e inequality,
$$
\frac 14\|\nu -\mu\|_{\textrm{TV}}^2 \le \Var_\mu(\sqrt{f})\le \SG
\EE(\sqrt{f}, \sqrt{f})=\SG I(\nu|\mu).
$$

\noindent $\bullet$ Statement (b). This converse part is based on
the third author's paper \cite{p-Wu01}. Indeed by Theorem
\ref{thm-a2.1} and our assumption,
$$
\limsup_{t\to\infty} \frac 1t \log\pp_\mu\left(\left|\frac
1t\int_0^t u(X_s)\,ds - \mu(u)\right|>r\right)<0
$$
for all $u:\XX\to\rr^d$ bounded and measurable ($d\ge 1$) and any
$r>0$. This implies by \cite[Theorem 3.9]{p-Wu01} that $(P_t)$
satisfies the Resolvent Tail-Norm Condition (named in
\cite{p-Wu01}). This last property together with the uniform
positive improving property implies the existence of a spectral
gap by \cite[Theorem 4.1]{p-Wu01} or \cite[Theorem 4.4]{GW06}.
\\
\noindent $\bullet$ Finally, (c) is a direct consequence of (a)
and (b). \nprf

\subsection{A CKP type inequality} During the proof of (a), we have used inequality (\ref{thm1a})
which is part of the Theorem \ref{thm3} below. The usual CKP
(Csisz\'ar-Kullback-Pinsker) inequality is
\begin{equation*}
    \frac 12 \|\nu-\mu\|_{\rm{TV}}^2 \le H(\nu|\mu),\ \forall
    \nu\in M_1(\XX).
\end{equation*}
We shall see during the proof of Lemma \ref{lem1} at (\ref{eq-12})
that (\ref{thm1a}) is of the similar form
\begin{equation*}
    \frac 14 \|\nu-\mu\|_{\rm{TV}}^2 \le I(\nu|\mu),\ \forall
    \nu\in M_1(\XX)
\end{equation*}
for some well chosen $I.$ This is the reason why it is called a
CKP type inequality.

Let $((X_t)_{t\ge0}, \mathbb{P} )$ be the pure jump Markov process
on the state space $\XX$ with generator $\mathcal{L}g(x)=\int_\XX
[g(y)-g(x)]\,\mu(dy),$ $g\in b\mathcal{B}$ and initial law $\mu\in
M_1(\XX).$ A representation of $X$ is given by
\begin{equation}\label{eq-10}
    X_t=Y_{N_t},\quad t\ge0
\end{equation}
where $N$ is a Poisson process with parameter one which is
independent of the sequence $(Y_n)_{n\ge0}$ of independent
identically $\mu$-distributed $\XX$-valued random variables.

\begin{thm}\label{thm3}
 Let $\mu$ be any probability measure on $\XX$ and $X$
be its associated process defined at (\ref{eq-10}).
\begin{enumerate}
    \item The following two equivalent families of inequalities hold
    true:
\begin{enumerate}[(a)]
    \item For every $\mu$-probability density $f,$ i.e.\! $f\ge 0$ and
$\mu(f)=1$,
\begin{equation}\label{thm1a}
\|f\mu -\mu\|_{\textrm{TV}}^2 \le 4\Var_\mu(\sqrt{f});
\end{equation}
    \item For all $\lambda\in\mathbb{R}$ and $u\in b\mathcal{B}$
    such that $\mu(u)=0$ and $\delta(u)\le2,$ we have
\begin{equation}\label{eq-01rho}
     \limsup_{t\rightarrow\infty}\frac
    1t\log\mathbb{E} \exp\left(\lambda\int_0^tu(X_s)\,ds\right)\le
    \rho(\lambda)
\end{equation}
where
$\rho(\lambda)=\mathbf{1}_{|\lambda|\le1}\lambda^2+\mathbf{1}_{|\lambda|>1}(2|\lambda|-1).$
\end{enumerate}
    \item The constant 4 in (\ref{thm1a}) is sharp and the equality holds if and only if
\begin{equation}\label{eq-02}
    \mu\circ f^{-1}=p\delta_{\frac{1-p}{p}}+(1-p)\delta_{\frac{p}{1-p}}
\end{equation}
for some $0<p<1.$

    \item The function $\rho$ is the best right-hand side for the
    inequality (\ref{eq-01rho}) and the equality is achieved for some $\lambda\in\mathbb{R}$
    and some $u\in b\mathcal{B}$ such that $\mu(u)=0$ and $\delta(u)\le2,$  if and only if
    there exists $0<p<1$ such that $$\lambda=1-2p:=\lambda(p)$$ and
\begin{equation}\label{eq-09}
    \mu\circ u^{-1}=p\delta_{2-2p}+(1-p)\delta_{-2p}.
\end{equation}
\end{enumerate}
\end{thm}

\proof  The statement of Theorem \ref{thm3} is simply a gathering
of Lemmas \ref{thm1} and \ref{lem1} below. These lemmas provide
two distinct proofs of inequality (\ref{thm1a}).
\endproof

\begin{remarks}\begin{rm} \
\begin{enumerate}[(i)]
    \item Note the symmetry $p\leftrightarrow 1-p$ in $(\ref{eq-02})$ and
the antisymmetry in $(\ref{eq-09}):$ $(\lambda,u)\leftrightarrow
(-\lambda,-u)$ gives $p\leftrightarrow 1-p.$

    \item Let us recall some well-known facts about optimal transportation \cite{Vill03, p-Vill05}. The total variation
$\|\mu-\nu\|_{\mathrm{TV}}$ is the minimal transportation cost
$T_c(\nu,\mu)$ associated with the cost function
$c(x,y)=2\cdot\mathbf{1}_{x\not=y}$ (see Remarks \ref{rem31}) :
\begin{equation*}
    \|\mu-\nu\|_{\mathrm{TV}}
    =2\inf_{\pi\in P(\nu,\mu)} \pi(\{(x,y); x\not=y\})
\end{equation*}
where $P(\nu,\mu)=\{\pi\in M_1(\XX^2):\pi_0=\nu, \pi_1=\mu\}.$ The
infimum is attained on $P(\nu,\mu),$ these minimizers are often
called optimal couplings of $\nu$ and $\mu.$ One has the following
characterization: $\pi\in P(\nu,\mu)$ is optimal if and only if
there exists some measurable function $u$ on $\XX$ such that
\begin{equation}\label{eq-11}
    \pi(\{(x,y)\in\XX^2; u(x)-u(y)=2\cdot\mathbf{1}_{x\not=y}\})=1.
\end{equation}
Such a $u$ is often called an optimal Kantorovich potential.
\\
  Let $0<p<1$ and $f$ satisfy $(\ref{eq-02}).$ Any optimal
    coupling $\pi$ of $f\mu$ and $\mu$ satisfies
    $$\pi(\{(x,y); x\not=y\})=|1-2p|=|\lambda(p)|$$
    and it admits an optimal
    Kantorovich potential $u$ (see (\ref{eq-11})) satisfying
    $(\ref{eq-09}).$ More precisely,
    \begin{enumerate}
        \item $\{u=2-2p\}=\{f=(1-p)/p\}$ if
    $0<p<1/2;$
        \item $\{u=2-2p\}=\{f=p/(1-p)\}$ if $1/2<p<1;$
        \item When $p=1/2,$ (\ref{eq-02}) is equivalent to $f=1$
        $\mu$-a.e., that is $\nu=\mu.$ On the other hand, any $u$ satisfying (\ref{eq-09}) with $p=1/2$ is an optimal
        potential for the trivial optimal coupling $\pi(dxdy)=\mu(dx)\delta_x(dy).$
    \end{enumerate}
    The equalities in (a) and (b) are satisfied up to
    $\mu$-negligible sets.

    \item Inequality (\ref{thm1a}) is already known in statistics. Indeed
for $\nu=f\mu$, the Hellinger distance between $\nu$ and $\mu$ is
given by
$$
d_H^2(\nu, \mu)=\frac 12 \int (1-\sqrt f)^2 d\mu = 1-\mu(\sqrt f).
$$
The known inequality (see Gibbs and Su \cite{GS02}) is
$$
\frac 14\|\nu -\mu\|_{\textrm{TV}}^2\le
d_H^2(\nu,\mu)[2-d_H^2(\nu,\mu)]
$$
and the above right-hand side is exactly
$1-[\mu(\sqrt{f})]^2=\Var_\mu(\sqrt{f})$.
\end{enumerate}
\end{rm}
\end{remarks}

We are going to give two  different new proofs of (\ref{thm1a}).
The first one, at Lemma \ref{thm1}, is elementary and provides a
characterization of these $f$'s which achieve the equality in
(\ref{thm1a}). The second one, at Lemma \ref{lem1}, is in the
spirit of this paper since it is a corollary of Theorem
\ref{thm-a2.1}. It also provides a characterization of the real
parameters $\lambda$ which achieve the equality in the ``dual"
inequality (\ref{eq-01rho}).

\subsubsection*{A first proof of the CKP type inequality (\ref{thm1a}) and more}
It may be seen as an amusing exercise in graduate courses.

\begin{lem}\label{thm1} The inequality (\ref{thm1a}) holds for every $\mu$-probability density
$f.$ Moreover, equality  is achieved in (\ref{thm1a}) if and only
(\ref{eq-02}) is satisfied for some $0<p<1.$
 \end{lem}

\bprf[Proof of Lemma \ref{thm1}] Assume that $\mu(f=1)<1$ (trivial
otherwise) in the following. Then $0<\mu(f<1), \mu(f>1)<1.$

\medskip\noindent $\bullet$\ \textsl{Step 1.  Reduction to the two-values case.}
This step might be the most difficult. Let $A=\{f<1\},
B=\{f\ge1\}$ and $\bar f:= \ee^\mu(f|\sigma(A))=\alpha
\mathbf{1}_A+\beta\mathbf{1}_B$ with
$\alpha=\mu(\mathbf{1}_Af)/\mu(A)$ and
$\beta=\mu(\mathbf{1}_Bf)/\mu(B).$ As $\mu(f<1)>0$ and
$\mu(f>1)>0,$ one sees that $0\le\alpha<1<\beta.$ Therefore, $f<1$
if and only if $\bar f<1$ and
$$
\frac 12 \mu(|1-f|)= \int_{f<1} (1-f)\,d\mu =  \int_{f<1} (1-\bar
f)\,d\mu =\int_{\bar f<1} (1-\bar f)\,d\mu= \frac 12 \mu(|1-\bar
f|)
$$
On the other hand, by Jensen's inequality
$$
\Var_{\mu}\left(\sqrt{\bar f}\right) = 1-\left[\mu\left(\sqrt{\bar
f}\right)\right]^2\le 1-[\mu(\sqrt{f})]^2
$$
and the equality holds if and only if $f=\bar f.$ It follows that
for $f$ to satisfy (\ref{thm1a}) it is enough that (\ref{thm1a})
also holds for $\bar f.$ Without loss of generality, we may assume
from now on that $f=\bar f$, i.e.\! there are two numbers $0\le
a<1<b$ such that
$$
p=\mu(f=a^2)\in (0,1), \ q=\mu(f=b^2)\in (0,1), p+q=1.
$$

\medskip\noindent $\bullet$\ \textsl{Step 2.}
Let $\xi\in (0, \pi/2)$ such that $\sqrt{p}=\cos \xi,
\sqrt{q}=\sin \xi$. Since
$$
1=\mu(f) = p a^2 + q b^2
$$
we may choose $\theta\in [0,\pi/2]$ such that
$$
a= \frac 1{\sqrt{p}} \cos \theta, \ b= \frac 1{\sqrt{q}} \sin
\theta.
$$
As $a\in [0,1)$, $\theta>\xi$. Now noting that
$$
\Var_\mu(\sqrt{f})= (b-a)^2 pq, \ \frac 12 \mu(|1-f|)=
\mu(1_{f<1}(1-f))=(1-a^2)p
$$
the inequality (\ref{thm1a}) amounts to saying that
\begin{equation} \aligned (b-a) \sqrt{pq} - (1-a^2)p &=\sin \theta \cos \xi -
\cos \theta \sin \xi - (\cos^2\xi -\cos^2\theta)\\
&=\sin(\theta-\xi)- (\cos^2\xi -\cos^2\theta)=:g(\xi, \theta)\ge
0\endaligned\end{equation} for all $0<\xi<\theta\le \pi/2$. Fix
$\xi$ and put $h(\theta)=g(\xi, \theta)$. We have
$$
h'(\theta)=\cos(\theta-\xi) - \sin (2\theta)= \sin\left(\frac \pi
2 - (\theta-\xi)\right)-\sin(2\theta)
$$
and then $h'(\theta)=0$ if and only if $\dsp \frac \pi 2 -
(\theta-\xi)=2\theta$ (if $2\theta\le \pi/2$) or $\dsp \frac \pi 2
- (\theta-\xi)=\pi-2\theta$ (if $2\theta\ge \pi/2$), i.e.\! $\dsp
\theta=\theta_1=\frac \pi 6 + \frac \xi 3$ or $\dsp
\theta=\theta_2=\frac \pi 2 - \xi$.

\begin{description}
\item[Case 1. $\xi\in [\pi/4,\pi/2)$]
    Since $\theta_1, \theta_2\le \xi$ and $h'(\pi/2)>0$, we
have $h'(\theta)>0$ for all $\theta\in (\xi, \pi/2]$. Hence for
all $\theta\in (\xi, \pi/2]$, $ g(\xi, \theta)=h(\theta)>
h(\xi)=0. $

\item[Case 2. $\xi\in (0, \pi/4)$] In this case
$\xi<\theta_1<\theta_2<\pi/2$. Since $ h'(\xi)>0, h'(\pi/2)>0, \
h(\xi)=h(\theta_2)=0, $  $h'$ is positive, negative and positive
respectively on $(\xi, \theta_1)$, $(\theta_1, \theta_2)$ and
$(\theta_2,\pi/2)$. Consequently for all $\theta\in (\xi, \pi/2],$
$ g(\xi, \theta)=h(\theta)\ge 0 $ and the equality holds if and
only if $\dsp \theta=\theta_2=\frac \pi 2 -\xi$.
\end{description}

\medskip\noindent $\bullet$\ \textsl{Step 3. Equality in (\ref{thm1a})}. If $f=1,
\mu$-a.s., then the equality in (\ref{thm1a}) holds. Now, assume
$\mu(f=1)<1$ and the equality in (\ref{thm1a}). By Step 1, $f=\bar
f$, i.e.\!  $f$ takes only two values $\sqrt{a}<1<\sqrt{b}$. By
Step 2, this is possible if and only if $\xi<\pi/4$ and $\theta=
\pi/2 -\xi$, i.e.\! $ p=\cos^2 \xi>\frac 12,$ $\cos^2\theta +
\cos^2 \xi= a^2 p + p=1.$ That is $p(1+a^2)=1.$ Therefore, either
there exist two numbers $0<a<1<b$ such that $\mu(f\in\{a^2,b^2\})=
1$ and
\begin{equation*}
    \left\{\begin{array}{lcl}
      pa^2+(1-p)b^2&=&1  \\
      p(1+a^2)&=&1 \\
    \end{array}\right.
\end{equation*}
or $f=1,\mu$-a.e. This proves the desired parametrization
(\ref{eq-02}) for $1/2\le p<1,$ and hence for all $0<p<1$ because
of the symmetry in (\ref{eq-02}). \nprf

\bexa {\rm (Bernoulli distribution). Let $\mu$ be the Bernoulli
distribution on $\XX=\{0,1\}$ with $\mu(\{1\})=p\in (0,1)$.
Consider the Dirichlet form $\EE(g,g)=(g(1)-g(0))^2$. Then
$\SG(\mu)=pq$. By Theorem \ref{thm2}-(a) and Remark
\ref{rem31}-(i), we see that
$$
pq W_1(\nu,\mu)^2\le I(\nu|\mu)
$$
where $W_1$ is built with the trivial metric. The constant $pq$ is
sharp. However $\mu$ does not satisfy any $W_2I(\kappa)$ as is
easily seen with $\nu=\mu_\vep=(1+\vep g)\mu$. } \nexa

\subsubsection*{A second proof of the CKP type inequality (\ref{thm1a}) and more}
Recall that $((X_t)_{t\ge0}, \mathbb{P} )$ is the pure jump Markov
process defined at (\ref{eq-10}).

\begin{lem}\label{lem1}\
\begin{enumerate}
    \item  The inequality (\ref{thm1a}) holds for all
probability density $f$ if and only if for all
$\lambda\in\mathbb{R}$ and
 $u\in b\mathcal{B}$ such that $\mu(u)=0$ and $\delta(u)\le2,$
we have
\begin{equation}\label{eq-01}
    \limsup_{t\rightarrow\infty}\frac
    1t\log\mathbb{E} \exp\left(\lambda\int_0^tu(X_s)\,ds\right)\le
    \lambda^2.
\end{equation}
    \item  For all $\lambda\in \mathbb{R}$
    and $u\in b\mathcal{B}$ such that $\mu(u)=0$ and $\delta(u)\le2,$
        the inequality (\ref{eq-01}) holds true.

    \item The equality is achieved in  (\ref{eq-01}) for some $\lambda\in\mathbb{R}$
    and some $u\in b\mathcal{B}$ such that $\mu(u)=0$ and $\delta(u)\le2,$  if and only if
     $\lambda=1-2p$ and (\ref{eq-09}) hold
for some $0<p<1.$

    \item The function $\rho$ is sharp in
    inequality (\ref{eq-01rho}), that is $$\rho(\lambda)=\sup_u\left\{\limsup_{t\rightarrow\infty}\frac
    1t\log\mathbb{E}
    \exp\left(\lambda\int_0^tu(X_s)\,ds\right)\right\}$$ for all real $\lambda,$ where the
    supremum is taken over all $u\in b\mathcal{B}$ such that $\mu(u)=0$ and $\delta(u)\le2.$

    \item For all $\lambda\in \mathbb{R}$
    and $u\in b\mathcal{B}$ such that $\mu(u)=0$ and $\delta(u)\le2,$
    the inequalities (\ref{thm1a}), (\ref{eq-01rho}) and (\ref{eq-01})
are equivalent. The equality in (\ref{eq-01rho}) is never achieved
whenever $|\lambda|\ge1.$
\end{enumerate}
\end{lem}

\begin{proof}
\noindent$\bullet$\ Statement (1): Clearly, $\mu$ is a reversing
measure for the process $X$ and the associated Dirichlet form is
\begin{equation}\label{eq-12}
    \mathcal{E}(g,g)=\Var_\mu(g),\quad g\in
    \mathbb{D}_2(\mathcal{L})=L^2(\mu).
\end{equation}
Therefore, statement (1) is a direct consequence of Theorem
\ref{thm-a2.1}-$(b')$ applied with $\Phi=\Phi_c,$
$c(x,y)=2.\mathbf{1}_{x\not=y}$  and $\alpha(a)=a^2/4.$ Note that
one passes from $\lambda\ge0$ to $\lambda\in\mathbb{R},$ by
considering $-u$ instead of $u.$ This is possible since
$\delta(-u)=\delta(u)\le2$ and $\mu(-u)=\mu(u)=0.$

\par\medskip\noindent$\bullet$\ Statement (2):
Let us introduce the notations $\psi(t)=\log\varphi(t)$ with
$\varphi(t)=\mathbb{E} \exp\left( \int_0^tv(X_s)\,ds\right)$ and
$v=\lambda u.$ We wish to get an upper bound for $\psi(t).$ For
all $t\ge0,$
$$
\begin{array}{ccccc}
  \psi'(t) & = & \frac{\varphi'}{\varphi}(t) & = & \mathbb{E}^{v,t}[v(X_t)] \\
  \psi''(t) & = & \left[\frac{\varphi''}{\varphi}-\left(\frac{\varphi'}{\varphi}\right)^2\right](t)
    &  &  \\
\end{array}
$$
where $\mathbb{E}^{v,t}$ is the expectation with respect to
$$\mathbb{P}^{v,t}=\exp\left(\int_0^tv(X_s)-\psi(t)\right)\cdot\mathbb{P}
.$$ In order to compute $\psi''(t),$ let us apply It\^o's formula
to $Y_t=v(X_t)\exp\int_0^tv(X_s)\,ds.$ This gives
\begin{equation*}
    dY_t=\exp\left(\int_0^tv(X_s)\,ds \right)\bigg[(\mathcal{L}v(X_t)+v(X_t)^2)\,dt + dM_t\bigg]
\end{equation*} where $M$ is some martingale. It follows that
\begin{eqnarray*}
  \varphi'(t)&=& \mathbb{E}Y_t\\
  &=& \mathbb{E} \left(v(X_0)+\int_0^t dY_s\right) \\
  &=& \mathbb{E}
v(X_0)+\int_0^t\mathbb{E}\left(\exp(\int_0^sv(X_r)\,dr)
[\mathcal{L}v(X_s)+v(X_s)^2]\right)\,ds
\end{eqnarray*}
Hence $\varphi''(t)=\mathbb{E}\left(\exp(\int_0^tv(X_s)\,ds)
[\mathcal{L}v(X_t)+v(X_t)^2]\right)$ and we obtain
\begin{equation*}
    \psi''(t)-\mathbb{E}^{\lambda u,t}\Big(\mathcal{L}[\lambda u](X_t)\Big)=\lambda^2\mathbb{V}^{\lambda u,t}[u(X_t)]
\end{equation*}
where $\mathbb{V}^{\lambda u,t}$ is the variance with respect to
$\mathbb{P}^{\lambda u,t}.$\\ Since $\mu(u)=0,$
$\mathcal{L}u(x)=\mu(u)-u(x)=-u(x),$  one sees that
\begin{equation}\label{eq-03}
    \psi''(t)+\psi'(t)=\lambda^2\mathbb{V}^{\lambda u,t}[u(X_t)],\quad t\ge0.
\end{equation}
As it is assumed that $\delta(u)\le2,$ we have
$\mathbb{V}^{\lambda u,t}[u(X_t)]\le1$ and
\begin{equation}\label{eq-04}
    \psi''(t)+\psi'(t)\le\lambda^2,\quad t\ge0.
\end{equation}
Clearly, $\psi(0)=0$ and $\psi'(0)=\mathbb{E}v(X_0)=\mu(v)=0.$ Let
$\theta$ be the solution of
\begin{equation}\label{eq-06}
    \theta''(t)+\theta'(t)=\lambda^2,\quad t\ge0
\end{equation}
with the same initial conditions as $\psi:$
$\theta(0)=\theta'(0)=0.$ Denote $\gamma(t)=\theta(t)-\psi(t).$ As
$\gamma''(t)+\gamma'(t)\ge0,$ for all $t\ge0,$ we see that
$\gamma'(t)+\gamma(t)\ge\gamma'(0)+\gamma(0)=0.$ By Gronwall's
inequality, it follows that $\gamma(t)\ge0$ for all $t\ge0.$ This
means that
\begin{equation*}
    \psi(t)\le \lambda^2 (e^{-t}+t-1),\quad \forall t\ge0
\end{equation*}
where the right hand side is $\theta(t)$ which is obtained by
elementary differential calculus. The upper bound (\ref{eq-01})
follows immediately.

\par\medskip\noindent$\bullet$\ Statement (3):
Now, let's investigate the equality in (\ref{eq-01}). Inspecting
the proof of statement (2), one sees that this equality holds if
and only if the equality holds in (\ref{eq-04}) asymptotically as
$t$ tends to infinity, i.e.
\begin{equation}\label{eq-07}
  \lim_{t\rightarrow \infty}  \mathbb{V}^{\lambda u,t}[u(X_t)]=1
\end{equation}
To see this, remark that the solution to (\ref{eq-06}) with
general initial conditions $\theta(0)=\theta_o,$
$\theta'(0)=\theta'_o$ is
$\theta(t)=\lambda^2(e^{-t}+t-1)+\theta_o+\theta'_o(1-e^{-t}),$
$t\ge0$ so that $\lim_{t\rightarrow\infty}\theta(t)/t=\lambda^2,$
for any initial conditions. Using Gronwall's inequality as above,
one sees now with (\ref{eq-03}) that the desired equality holds if
and only if $\liminf_{t\rightarrow \infty} \mathbb{V}^{\lambda
u,t}[u(X_t)]=1.$ But this is equivalent to (\ref{eq-07}) since
$\mathbb{V}^{\lambda u,t}[u(X_t)]\le1$ as $\delta(u)\le2.$
\\
It is an easy exercice to show that for any random variable $Z$
such that $-1\le Z\le1$ almost surely, we have $\Var(Z)\le1$ and
$\Var(Z)=1$ if and only if the law of $Z$ is $\frac 12
(\delta_{-1}+\delta_{+1}).$ It follows that (\ref{eq-07}) holds if
and only if $u$ only takes two values $a+1$ and $a-1$ where $a$ is
some real such that $\mu(u)=0$ and
\begin{equation}\label{eq-08}
 \lim_{t\rightarrow\infty}\mathbb{P}^{\lambda
u,t}(u(X_t)=a+1)=\lim_{t\rightarrow\infty}\mathbb{P}^{\lambda
u,t}(u(X_t)=a-1)=1/2.
\end{equation}
One immediately sees that with $p:=\mu(u=a+1),$ $\mu(u)=0$ implies
that $a=a(p)=1-2p.$ Therefore, the image law $\mu\circ u^{-1}$ of
$\mu$ by $u$ must satisfy (\ref{eq-09}) for some $0<p<1.$
\\
Fix $0<p<1$ and take $u$ as in (\ref{eq-09}). Let us introduce the
$\{+,-\}$-valued process defined by $Z_t^{(p)}=u(X_t)-a(p),$
$t\ge0$  and denote $\chi$ the identity on $\{+,-\}:$
$\chi(+)=+1,$ $\chi(-)=-1.$ Rewriting (\ref{eq-08}) as a ratio,
one obtains
\begin{eqnarray*}
 1&=&
    \lim_{t\rightarrow\infty}\frac
{\mathbb{E}\left[\mathbf{1}_{\{u(X_t)=a(p)+1\}}\left\{\exp\left(\lambda\int_0^tu(X_s)\,ds\right)-\psi(t)\right\}\right]}
{\mathbb{E}\left[\mathbf{1}_{\{u(X_t)=a(p)-1\}}\left\{\exp\left(\lambda\int_0^tu(X_s)\,ds\right)-\psi(t)\right\}\right]}\\
 &=&  \lim_{t\rightarrow\infty}\frac
 {\mathbb{E}\left[\mathbf{1}_+(Z_t^{(p)})\exp\left(\lambda\int_0^t\chi(Z_s^{(p)})\,ds\right)\right]}
 {\mathbb{E}\left[\mathbf{1}_-(Z_t^{(p)})\exp\left(\lambda\int_0^t\chi(Z_s^{(p)})\,ds\right)\right]}\\
\end{eqnarray*}
The process $Z^{(p)}$ is still  Markov and its generator is given
for all $\sigma\in\{+,-\}$ and $g\in\mathbb{R}^{\{+,-\}}$ by
\begin{equation*}
    A_pg(\sigma)=m_p(g)-g(\sigma)
\end{equation*}
where $m_p=p\delta_++(1-p)\delta_-$ is the image of $\mu$ by
$x\in\XX\mapsto u(x)-a(p).$
\\
The Feynman-Kac semigroup associated with $Z^{(p)}$ and the
potential $\lambda\chi$ is defined for all $\sigma\in\{+,-\}$ and
$g\in\mathbb{R}^{\{+,-\}}$ by
\begin{equation*}
    P_t^\lambda
g(\sigma):=\mathbb{E}\left[g(Z_t^{(p)})\exp\left(\lambda\int_0^t\chi(Z_s^{(p)})\,ds\right)\Big
| Z_0^{(p)}=\sigma\right].
\end{equation*}
This allows to rewrite
\begin{eqnarray*}
    \frac
    {\mathbb{E}\left[\mathbf{1}_+(Z_t^{(p)})\exp\left(\lambda\int_0^t\chi(Z_s^{(p)})\,ds\right)\right]}
 {\mathbb{E}\left[\mathbf{1}_-(Z_t^{(p)})\exp\left(\lambda\int_0^t\chi(Z_s^{(p)})\,ds\right)\right]}
    &=&\frac
    {\langle m_p,P_t^\lambda \mathbf{1}_+\rangle}
    {\langle m_p,P_t^\lambda\mathbf{1}_-\rangle}\\
    &=&\frac
    {\langle m_p,\exp[t(A_p+\lambda\chi)] \mathbf{1}_+\rangle}
    {\langle m_p,\exp[t(A_p+\lambda\chi)]\mathbf{1}_-\rangle}\\
\end{eqnarray*}
which can be computed explicitly by means of elementary linear
algebra in $\mathbb{R}^2.$ Indeed, seeing the functions as
vectors:
 $g=\begin{pmatrix} g(+) \\ g(-)\end{pmatrix}
 = \begin{pmatrix}x\\y\end{pmatrix}
 =x\mathbf{1}_++y\mathbf{1}_-
 $
where $(\mathbf{1}_+,\mathbf{1}_-)$ is the canonical base of
$\mathbb{R}^{\{+,-\}},$ one
immediately identifies the operator $A_p+\lambda\chi$ with the matrix $M=\left(%
\begin{array}{cc}
  -1+p+\lambda & 1-p \\
  p & -p-\lambda \\
\end{array}%
\right).$ It is a simple exercice to show that $M$ has two real
distinct eigenvalues $s_1>s_2$ given by $s_1=-1/2
+\sqrt{\Delta}/2$ and $s_2=-1/2-\sqrt{\Delta}/2$ with
$\Delta=1+4\lambda(\lambda+2p-1)$ and $\Delta>0$ since $0<p<1.$
One can check that $\mathbf{1}_+=(v_1-v_2)/\sqrt{\Delta}$ and
$\mathbf{1}_-=(-\frac{p+\lambda+s_2}{p}v_1+\frac{p+\lambda+s_1}{p}v_2)/\sqrt{\Delta}$
where $v_1=\begin{pmatrix} p+\lambda+s_1 \\ p\end{pmatrix}$ and
$v_2=\begin{pmatrix} p+\lambda+s_2 \\ p\end{pmatrix}$ are
eigenvectors associated respectively with $s_1$ and $s_2.$ Using
the elementary remark that $s_1>s_2$ implies that
$\lim_{t\rightarrow\infty} e^{-ts_1}e^{tM}(av_1+bv_2)=av_1$ for
all $a,b\in\mathbb{R},$ one obtains
\begin{equation*}
    \lim_{t\rightarrow\infty}\frac
    {\langle m_p,P_t^\lambda \mathbf{1}_+\rangle}
    {\langle m_p,P_t^\lambda\mathbf{1}_-\rangle}
    = -\frac{p}{p+\lambda+s_2}
\end{equation*}
for all $0<p<1$ and $\lambda\in\mathbb{R}.$ Putting everything
together, we conclude that the equality in (\ref{eq-01}) holds if
and only if $-p/(p+\lambda+s_2)=1$ which in turn is also
equivalent to $\lambda=1-2p.$

\par\medskip\noindent$\bullet$\ Statement (4):
Let us denote
$$R(\lambda):=\sup_u\left\{\limsup_{t\rightarrow\infty}\frac
    1t\log\mathbb{E}
    \exp\left(\lambda\int_0^tu(X_s)\,ds\right)\right\},\ \lambda\in \mathbb{R}$$
where the supremum is taken over all $u\in b\mathcal{B}$ such that
$\mu(u)=0$ and $\delta(u)\le2.$ We have to show that $R=\rho.$
\\
Let us first prove that $R$ is a convex function.  As a
log-Laplace transform, for each $u$ and $t,$ $\log\mathbb{E}
\exp\left(\lambda\int_0^tu(X_s)\,ds\right)$ is a convex function
of $\lambda.$ Since $X$ has stationary independent increments, by
a standard sub-additivity argument, one shows that
$\limsup_{t\rightarrow\infty}$ is a genuine limit:
$\lim_{t\rightarrow\infty}.$ It follows that for each $u,$
$\limsup_{t\rightarrow\infty}\frac
    1t\log\mathbb{E}
    \exp\left(\lambda\int_0^tu(X_s)\,ds\right)$
is a convex function of $\lambda.$ Finally, $R$ is convex as it is
the supremum of convex functions.
\\
Because of statements (2) and (3), it is already seen that
$R(\lambda)=\rho(\lambda)=\lambda^2$ for all $-1<\lambda<1.$ On
the other hand, since $\sup|u|\le2$ for each $u$ such that
$\mu(u)=0$ and $\delta(u)\le2,$ it is clear that
$\limsup_{t\rightarrow\infty}\frac 1t\log\mathbb{E}
\exp\left(\lambda\int_0^tu(X_s)\,ds\right) \le 2|\lambda|$ for all
real $\lambda.$ Therefore, for all $\lambda,$ $\rho(\lambda)\le
R(\lambda)\le \mathrm{cv\,} r(\lambda)$ where $\mathrm{cv\,} r$ is
the convex envelope of $r(\lambda)= \min(\lambda^2,2|\lambda|).$
Indeed, the first inequality holds since $\rho$ is the lowest
convex function which matches with $\lambda^2$ on $\lambda\in
(-1,1)$ while the second one follows from the inequality
$R(\lambda)\le \min(\lambda^2,2|\lambda|)$ for all $\lambda$ and
the convexity of $R.$ One concludes that $R=\rho,$ remarking that
$\mathrm{cv\,} r=\rho.$

\par\medskip\noindent$\bullet$\ Statement (5) is a direct
consequence of statement (1) and the proof of statement (4).
\end{proof}

Note that $W_1I$ for the trivial metric implies $W_1I$ for any
bounded metric. So our next purpose is to obtain  $W_1I$ for
unbounded metrics. Our study is naturally separated into two
sections. Next Section 4 is concerned with estimating sharply
$\WI$ under strong dissipative conditions. In Section 5, Lyapunov
function conditions for $W_1I$ or more general $T_\Phi I$ are
taken into consideration.

%%%%%%%%%%%%%%%%%%%%%%%%%%%%
%%%%%%%%%%%%%%%%%%%%%%%%%%%%

\section{Spectral gap in the space of Lipschitz functions implies $W_1I$
for diffusion processes}

\subsection{General observations}
We begin with the particular case where $\mu(dx)= e^{-V(x)}dx/Z$
($Z$ is the normalization constant) with $V\in C^2(\XX)$ on a
connected and complete Riemannian manifold $\XX$, the diffusion
$(X_t)$ generated by $\LL =\Delta -\nabla V\cdot \nabla$
($\Delta,\nabla$ are respectively the Laplacian and gradient on
$M$) is reversible with respect to $\mu$, and the corresponding
Dirichlet form is given by
$$
\EE(h,h)=\int_\XX |\nabla h|^2 \,d\mu, \ \forall h\in
\dd(\EE)=H^1(\XX, \mu).
$$

\begin{thm} \label{thm41} Assume that $\int_\XX d^2(x,x_0)\,d\mu(x)<+\infty$
and $\LL$ has a spectral gap on the space $C_{\mathrm{Lip}}(\XX)$
of Lipschitz functions with respect to the Riemannian metric $d$,
i.e.\!  there is a best finite constant $C>0$ such that for any
$g\in C_{\mathrm{Lip}}\bigcap b\BB$ with $\mu(g)=0$, there is
$h\in \dd_2(\LL)$ with $\mu(h)=0$ solving the Poisson equation
$$
-\LL h = -\Delta h + \nabla V\cdot \nabla h =g, \
\mu\textrm{-a.s.}
$$
such that one $\mu$-version $\tilde h$ of $h$ verifies
\begin{equation}\label{thm41a1} \|\tilde h\|_{\mathrm{Lip}}\le
C\|g\|_{\mathrm{Lip}}. \end{equation} Then $\mu$ satisfies
$W_1I(C)$:
\begin{equation}\label{thm41a} W_1(\nu,\mu)^2\le 4C^2\, I(\nu|\mu),\
\forall \nu\in M_1(\XX) \end{equation} or equivalently for any
Lipschitz function $u$ on $\XX$ and any initial probability
measure $\beta$ with $d\beta/d\mu\in L^2(\mu)$
\begin{equation}\label{thm41b} \pp_\beta\left(\frac 1t\int_0^t u(X_s)\,ds\ge
\mu(u)+r\right) \le
\left\|\frac{d\beta}{d\mu}\right\|_2\exp\left(-t\frac{
r^2}{4C^2\|u\|_{\mathrm{Lip}}^2}\right),\ \forall r,t>0.
\end{equation}
 \end{thm}

\brmk{\rm\
\begin{enumerate}[(i)]
    \item
Let $C_{\mathrm{Lip}}^0$ be the Banach space of those $g\in
C_{\mathrm{Lip}}$ with $\mu(g)=0$, equipped with
$\|\cdot\|_{\mathrm{Lip}}$. Hence the best constant $C$ in
(\ref{thm41a1}) is exactly
$$
\|(-\LL)^{-1}\|_{C_{\mathrm{Lip}}^0}.
$$
By the spectral decomposition we always have (cf. \cite[Proof of
Lemma 4.3]{Wu04})
$$
C=\|(-\LL)^{-1}\|_{C_{\mathrm{Lip}}^0} \ge
\|(-\LL)^{-1}\|_{L^2(\mu)\bigcap \{g\in L^2(\mu);
\mu(g)=0\}}=\SG(\mu).
$$
    \item
The constant in the concentration inequality (\ref{thm41b}) is
sharp. Indeed let $dX_t=\sqrt{2}\,dB_t - X_t \,dt$, which is
reversible with respect to $\mu=\NN(0,1)$ on $\rr$. For this model
we have $\nabla P_t=e^{-t}P_t\nabla$ and then
$$
\nabla (-\LL)^{-1}g = \nabla \int_0^\infty P_tg\, dt= (1-\LL)^{-1}
\nabla g, \forall g\in C_b^\infty(\rr), \mu(g)=0.
$$
This implies $C=1$ in (\ref{thm41a1}). On the other hand for
$u(x)=x$, under $\pp_\mu$, the law of
$$
\frac 1t\int_0^t u(X_s)\,ds =\frac 1t\int_0^t X_s \,ds
$$
is $\NN(0, \sigma^2(t))$ where the variance $\sigma^2(t)$ is given
by
$$
\sigma^2(t)=\frac 2{t^2} \iint_{0\le a\le b\le t} \ee_\mu X_a X_b
\,da db =\frac 2{t^2} \iint_{0\le a\le b\le t} e^{-(b-a)}
\,dadb=\frac 2t - \frac 2{t^2}(1-e^{-t})
$$
from which we get
$$
\lim_{t\to\infty}\frac 1t \log \pp_\mu\left(\frac 1t\int_0^t X_s
\,ds>r\right)= - \frac {r^2}4.
$$
This coincides with the upper bound $-r^2/(4C^2)$ derived from
(\ref{thm41b}), showing the sharpness of (\ref{thm41b}).
\end{enumerate}
}\nrmk

\bprf[Proof of Theorem \ref{thm41}] Let $\Phi=\{(g,g);
\|g\|_{\mathrm{Lip}}\le 1, g \text{ bounded}\}$. Then
$W_1(\nu,\mu)=T_\Phi(\nu,\mu)$ by Kantorovich-Rubinstein's
theorem. Let us verify that $(b')$ of Theorem \ref{thm-a2.1}
holds.
\\
For any $g\in C_{\mathrm{Lip}}$ with $\|g\|_{\mathrm{Lip}}\le 1$,
let $h\in C_{\mathrm{Lip}}\bigcap \dd_2(\LL)$ such that $-\LL
h=g$. Hence
$$
M_t(h):= h(X_t) - h(X_0)+\int_0^t g(X_s)\,ds
$$
and
$$
M^*_t(h):= h(X_0) - h(X_t) + \int_0^t g(X_s)\,ds
$$
have the same law under $\pp_\mu$ by the reversibility  of
$((X_t), \pp_\mu)$. Consequently from Lyons-Meyer-Zheng's
forward-backward martingale decomposition

\begin{equation}\label{forward-backward}
S_t(g):=\int_0^t g(X_s)\,ds=\frac 12(M_t(h)+M_t^*(h)),
\end{equation}
it follows that for any convex function $\phi$ on $\rr$,
$$
\ee_\mu \phi(S_t(g)) \le \frac 12
\ee_\mu[\phi(M_t(h)+\phi(M_t^*(h))]= \ee_\mu \phi(M_t(h))
$$
As $M_t(h)$ is a (forward) continuous martingale,
$M_t(h)=B_{\tau_t}$ where $(B_t)$ is some Brownian motion with
respect to another time-changed filtration $(\hat\FF_t)$, and
$\tau_t=\<M(h)\>_t$ is a $(\hat\FF_t)$-stopping time (a well known
result). Since
$$
\<M(h)\>_t=2\int_0^ t |\nabla h|^2(X_s)\,ds\le 2C^2 t
$$
By Jensen's inequality, we obtain for all convex function $\phi$
on $\rr$ that
\begin{equation}\label{convex-control}
\ee_\mu\phi(S_t(g))\le \ee\phi(B_{\tau_t})= \ee\phi(\ee[B_{2C^2
t}|\hat \FF_{\tau_t})\le \ee \phi(B_{2C^2 t})
\end{equation}
Applying this to $\phi(x)=e^{\lambda x}$, we get
\begin{equation}\label{thm41d}
\ee_\mu \exp\left(\lambda\int_0^ t g(X_s)\,ds\right) \le \ee
e^{\lambda B_{2C^2t}}= e^{\lambda^2C^2t}, \ \forall \lambda\in\rr.
\end{equation}
Hence Theorem \ref{thm-a2.1}-$(b')$ holds with $\Phi=\{(g,g);
\|g\|_{\mathrm{Lip}}\le 1, g \text{ bounded}\}$ and
$\alpha(r)=r^2/(4C^2).$ Therefore (\ref{thm41a}) and
(\ref{thm41b}) follow from Theorem \ref{thm-a2.1}.
 \nprf

Klein-Ma-Privault \cite{KMP06} developed {\it convex concentration
inequality} (\ref{convex-control}) for semimartingales instead of
$S_t(g)$, by means of a forward-backward martingale calculus, but
their result cannot be applied directly here.

Before estimating the constant $C$ in condition (\ref{thm41a1}),
we extend the above result to general symmetric Markov diffusions
by following Bakry \cite{Ba92}.

Let $((X_t), \pp_\mu)$ be a reversible ergodic Markov process with
the Dirichlet form $(\EE, \dd(\EE))$, with continuous sample paths
valued in some separable complete metric space $(\XX,d)$ (called
Markov diffusion). We assume that
 $(\EE, \dd(\EE))$ is given by the carr\'e-du-champs
$\Gamma: \dd(\EE)\times \dd(\EE)\to L^1(\mu)$ (symmetric, bilinear
definite nonnegative form):
\begin{equation}\label{Gamma}
\EE(h,h)=\int_\XX \Gamma(h,h)\,d\mu,\ \forall h\in\dd(\EE).
\end{equation}
The continuity of sample paths of $(X_t)$ implies that $\Gamma$ is
a differentiation, that is: for all $(h_k)_{1\le k\le n}\subset
\dd(\EE), g\in \dd(\EE)$ and $F\in C_b^1(\rr^n)$,
$$
\Gamma(F(h_1,\cdots, h_n), g) = \sum_{i=1}^n
\partial_iF(h_1,\cdots, h_n)\Gamma(h_i, g).
$$

With exactly the same proof as that of Theorem \ref{thm41} we have

\begin{thm}\label{thm42} Assume that $\int_\XX d^2(x,x_0)\,d\mu(x)<+\infty$
and for any $g\in C_{\mathrm{Lip}}(\XX,d)$ bounded with
$\mu(g)=0$, then $g\in \dd(\EE)$ and
\begin{equation}\label{thm42a}
\sqrt{\Gamma(g,g)}\le \sigma \|g\|_{\mathrm{Lip}}, \
\mu\textrm{-a.s.}
\end{equation}
and there is some $h\in \dd_2(\LL)$ such that $-\LL h=g$
($\mu$-a.e.) and a $\mu$-continuous version $\tilde h$ of $h$
satisfying
\begin{equation}\label{thm42b}
\|\tilde h\|_{\mathrm{Lip}} \le C \|g\|_{\mathrm{Lip}}
\end{equation}
where $\sigma, C>0$ are fixed constants. Then for any $u\in
C_{\mathrm{Lip}}(\XX,d)$ and any convex function $\phi$ on $\rr$,
$$
\ee_\mu \phi(S_t(g)) \le \ee \phi(B_{2\sigma^2 C^2 t})
$$
where $B$ is a standard Brownian Motion. In particular
$$
\ee_\mu \exp\left(\lambda\int_0^ t g(X_s)\,ds\right) \le
e^{\lambda^2 (\sigma C)^2\|g\|_{\mathrm{Lip}}^2 t}, \ \forall
\lambda\in\rr, t>0.
$$
and $\mu$ satisfies $W_1I(\sigma C)$ on $(\XX,d)$.
\end{thm}

\subsection{Multi-dimensional diffusions}

Let us show now how to estimate the constant $C$ in
(\ref{thm41a1}) or (\ref{thm42b}) by means of some examples.

\subsubsection*{A first example}
At first in the framework of Theorem \ref{thm41}, if the
Bakry-Emery curvature is positive
$$
\mathrm{Ric} + \nabla ^2V \ge K>0
$$
then it is well known that for $g\in C^1_b(\XX)$ with $\mu(g)=0$,
$$
|\nabla P_t g| \le e^{-K t} P_t|\nabla g|
$$
and then $h:=\int_0^\infty P_t g\, dt$ is absolutely convergent in
$C_{\mathrm{Lip}}\subset L^2(\mu)$ (for $\int_\XX
d^2(x,x_0)\,d\mu<+\infty$). Hence $h\in \dd_2(\LL)$, $-\LL h=g$
and
$$
\|h\|_{\mathrm{Lip}} \le \|\nabla g\|_\infty\int_0^\infty
e^{-Kt}\, dt  = \frac 1K \|g\|_{\mathrm{Lip}}.
$$
In other words condition (\ref{thm41a1}) holds with $C=K^{-1}$ and
 $\mu$ satisfies $W_1I(K^{-1})$. Of course one can also derive this
sharp  transportation inequality from the log-Sobolev inequality
of Bakry-Emery \cite{BE85} and Proposition \ref{prop25}-(b).

\subsubsection*{A second example}
Now we turn to another situation where the log-Sobolev inequality
is unknown as in Djellout et al.\! \cite{DGW03}. Consider the
stochastic differential equation

\begin{equation}\label{a44} dX_t =\sqrt{2} \sigma(X_t)\,dB_t + b(X_s)\,ds \end{equation}
where $\sigma : \rr^d\to\MM_{d\times n}$ (the space of real
$d\times n$-matrices) and $b: \rr^d\to \rr^d$ are locally
Lipschitz, and $(B_t)$ is a standard Brownian motion in $\rr^n$.
Assume that for some $\delta>0$,
\begin{equation}\label{45}
\mathrm{tr}[(\sigma(y)-\sigma(x))(\sigma(y)-\sigma(x))^T] + \<y-x,
b(y)-b(x)\> \le - \delta |y-x|^2, \ \forall x,y\in\rr^d
\end{equation}
Here $\mathrm{tr}(\cdot)$ denotes the trace and $A^T$ the
transposition of matrix $A.$ With It\^o's formula one easily
obtains
$$
\ee |X_t(y)-X_t(x)|^2\le e^{-2\delta t} |x-y|^2, \ \forall x,y\in
\rr^d, t\ge0
$$
where $X_t(x)$ is the solution of (\ref{a44}) with $X_0=x.$ This
implies that $(X_t)$ has a unique invariant probability measure
$\mu$. Hence for any $g\in C_{\mathrm{Lip}}(\rr^d)$ (with respect
to the Euclidean norm $|\cdot|$) with $\mu(g)=0$,

\begin{equation}\label{46}
\|P_t g\|_{\mathrm{Lip}} = \sup_{x\ne y} \frac{|\ee g(X_t(y)) -\ee
g(X_t(x))|}{|y-x|} \le e^{-\delta t} \|g\|_{\mathrm{Lip}}.
\end{equation}
Then $h:=\int_0^\infty P_t g\,dt$ is absolutely convergent in
$C_{\mathrm{Lip}}(\rr^d)$ and $\|h\|_{\mathrm{Lip}} \le
\delta^{-1} \|g\|_{\mathrm{Lip}}$. In other words (\ref{thm42b})
holds with $C=\delta^{-1}$. Finally as the carr\'e-du-champ of
$(X_t)$ is given by
$$
\Gamma(h,h)(x) =\<\sigma\sigma^T\nabla h(x), \nabla h(x)\>,\ h\in
C^2_0(\rr^d)
$$
the constant $\sigma$ in (\ref{thm42a}) can be identified as
$\|\sigma(\cdot)\|_\infty:=\sup_x\|\sigma(x)\|_{\rr^n\to\rr^d}$,
at least for $h\in C_0^2(\rr^d).$

\bcor\label{cor43} Assume that $\sigma, b$ are locally Lipschitz
such that $\|\sigma(\cdot)\|_\infty<+\infty$ and satisfy the
dissipativity condition (\ref{45}). Suppose moreover that $\mu$ is
absolutely continuous and its transition semigroup $(P_t)$ is
symmetric with respect to the unique invariant measure $\mu$. Then
the Dirichlet form $(\EE, \dd(\EE))$ on $L^2(\mu)$ is given by the
closure of
$$
\EE(h,h)=\int_{\rr^d} \<\sigma\sigma^T \nabla h, \nabla h\>
\,d\mu,\quad h\in C_0^\infty(\rr^d)
$$
and $\mu$ satisfies $W_1I(c)$ on $\rr^d$ with respect to the
Euclidean metric with
$$
c=\|\sigma(\cdot)\|_\infty/\delta.
$$
\ncor

\bprf As $C_{\mathrm{Lip}}\bigcap\dd_2(\LL)$ is stable by $(P_t)$
and contains $C_0^\infty(\rr^d)$,
$C_{\mathrm{Lip}}\bigcap\dd_2(\LL)$ is an operator core for $(\LL,
\dd_2(\LL))$, hence a form core for $(\EE, \dd(\EE))$. Since any
$h\in C_{\mathrm{Lip}}$ can be approached by a sequence $h_n$ in
$C_0^\infty(\rr^d)$ with respect to the norm
$$
\sqrt{\|h\|_2^2 + \int_{\rr^d} \<\sigma\sigma^T \nabla h, \nabla
h\> \,d\mu}
$$
one sees that $C_{\mathrm{Lip}}\subset \dd(\EE),$
$$
\EE(h,h) =  \int_{\rr^d} \<\sigma\sigma^T \nabla h, \nabla h\>
\,d\mu, \ \forall h\in C_{\mathrm{Lip}}
$$
and $C_0^\infty(\rr^d)$ is a form core for $(\EE, \dd(\EE))$. This
proves the first claim. It also follows that the conditions in
Theorem \ref{thm42} are verified. The remaining part follows from
Theorem \ref{thm42}. \nprf

\medskip\noindent \textbf{Remark.}\  The crucial formula (\ref{46}) for estimating our
condition (\ref{thm42b}) is equivalent to
\begin{equation}
W_1(P_t(x,\cdot), P_t(y,\cdot)) \le e^{-\delta t} d(x,y)
\end{equation}
which can obtained by means of numerous coupling techniques, see
M.F.~Chen \cite{Chen05}. Note that this condition is the one
introduced by Joulin \cite{JouPhD} under the name of Wasserstein's
curvature, with which he obtains Poisson type deviation
inequality.

\subsection{One-dimensional diffusions}

Now let us consider one-dimensional diffusion processes with
values in the interval $(x_0,y_0)$ and generated by
$$
\LL h = a h^{\prime\prime}+b h',\quad  h\in C_0^\infty(x_0,y_0)
$$
where $a,b$ are continuous with $a>0.$ Let $((X_t)_{0\le
t<\tau},\pp_x)$ be the martingale solution associated with $\LL$
and initial position $x$, where $\tau$ is the explosion time. With
a fixed $c\in (x_0,y_0)$,
$$
s'(x):=\exp\left(-\int_c^x \frac {b(z)}{a(z)}\, dz\right),\
m'(x):=\frac 1{a(x)} \exp\left(\int_c^x \frac {b(z)}{a(z)}
\,dz\right)
$$
 are respectively the derivatives of Feller's scale and speed
functions. Assume that
\begin{equation}\label{D1}
Z:=\int_{x_0}^{y_0} m'(x)\,dx<+\infty
\end{equation}
and let $\mu(dx)=m'(x)dx/Z$. It is well known that $(\LL,
C_0^\infty(x_0,y_0))$ is symmetric on $L^2(\mu)$.
\\
Assume also that

\begin{equation}\label{D2}
\int_c^{y_0} s'(x)\,dx \int_c^x m'(x)\,dx=\int_{x_0}^c
s'(x)\int_x^c m'(x)\,dx=+\infty
\end{equation}
which, in Feller's classification, means that $x_0$ and $y_0$ are
not accessible or equivalently $\tau=\infty$, $\pp_x$-a.s. In this
case by the $L^1$-uniqueness in \cite{Wu99, Ebe99}, the Dirichlet
form
\begin{eqnarray*}
\dd(\EE)&=&\left\{h\in\AA\CC(x_0, y_0)\bigcap L^2(\mu);
\int_{x_0}^{y_0}(h')^2\,d\mu<+\infty\right\}, \\
\EE(h,h)&=&\int_{x_0}^{y_0}(h')^2\,d\mu,\quad h\in\dd(\EE)
\end{eqnarray*}
is associated with $(X_t)$, where $\AA\CC(x_0, y_0)$ is the space
of the absolutely continuous functions on $(x_0,y_0)$.
\\
Fix some $\rho\in C^1(x_0,y_0)$ such that $\rho\in L^2(\mu)$ and
$\rho'(x)>0$ everywhere, consider the metric
$d_\rho(x,y)=|\rho(x)-\rho(y)|$. A function $h$ on $(x_0,y_0)$ is
Lipschitz with respect to $d_\rho$ (one writes $h\in
C_{\mathrm{Lip}(\rho)}$) if and only if $h\in\AA\CC(x_0,y_0)$ and
$$
\|h\|_{\mathrm{Lip}(\rho)}=\sup_{x_0<x<y<y_0}
\frac{|h(y)-h(x)|}{\rho(y)-\rho(x)}=\left\|
{h'}/{\rho'}\right\|_\infty<\infty
$$
The argument below is borrowed from \cite{p-DMW}. Assume that
\begin{equation}\label{C-rho}
C(\rho):=\sup_{x\in (x_0,y_0)} \frac 1{\rho'(x)} \int_x^{y_0}
[\rho(z)-\mu(\rho)] m'(z)\,dz<+\infty.
\end{equation}
For every $g\in C_{\mathrm{Lip}(\rho)}$ with $\mu(g)=0$, then the
$C^2$ function $ h(x)=\int_c^x dy \int_y^{y_0} g(z) m'(z)\,dz - A
$ solves
\begin{equation}\label{cor44a}
-(a h^{\prime\prime} + b h') = g.
\end{equation}
It is obvious that $ \|h\|_{\mathrm{Lip}(\rho)}=\sup_{x\in
(x_0,y_0)} \frac 1{\rho'(x)} \int_x^{y_0} g(z) m'(z)\,dz. $ An
elementary exercise (see \cite{p-DMW}) shows that this quantity is
not greater than $C(\rho)\|g\|_{\mathrm{Lip}(\rho)}$. Thus $h$
belongs to $L^2(\mu)$ whenever $\rho$ is in $L^2(\mu)$. By It\^o's
formula, $h\in \dd_2(\LL)$.   With the constant $A$ so that
$\mu(h)=0$, because of the ergodicity of $(X_t),$ $h$ is the
unique solution of (\ref{cor44a}) in $L^2(\mu)$ with zero mean.
One also sees that $C(\rho)$ is the best constant by taking
$g=\rho-\mu(\rho)$. In other words, condition (\ref{thm42b}) is
satisfied with $C=C(\rho)$. Hence, with Theorem \ref{thm42} one
obtains

\bcor\label{cor44} Let $a,b: (x_0,y_0)\to \rr$ be continuous with
$a>0$ and let conditions (\ref{D1}) and (\ref{D2}) be satisfied.
Assume (\ref{C-rho}) and $\sigma:=\sup_{x\in (x_0,y_0)}
\sqrt{a(x)} \rho'(x)<+\infty$. Then $\mu$ satisfies $W_1I(\kappa)$
on $((x_0,y_0), d_{\rho})$ with $\kappa=(\sigma C(\rho))^{-1}$. In
particular for
$$
\rho_a(x)=\int_c^x \frac {dz}{\sqrt{a(z)}}
$$
 if $C(\rho_a)<+\infty$, then $\mu$  satisfies
$W_1I( c)$ on $((x_0,y_0), d_{\rho_a})$ with $c=C(\rho_a)$. \ncor

\brmk{\rm\
\begin{enumerate}
    \item $d_{\rho_a}$ is the metric associated with the carr\'e-du-champ
operator of the diffusion.
    \item The quantity $C(\rho)$ in (\ref{C-rho}) is not innocent:
Chen-Wang's variational formula for the spectral gap tells us that
(\cite{Chen05, Wang}): $ \SG(\mu)=\inf_{\rho}
 C(\rho). $
\end{enumerate} } \nrmk

%%%%%%%%%%%%%%%%%%%%%%%%%%%%%%%%
%%%%%%%%%%%%%%%%%%%%%%%%%%%%%%%%
%%%%%%%%%%%%%%%%%%%%%%%%%%%%%%%%

\section{Lyapunov function conditions}

We will use in this section general conditions on the generator of
the process, known as Lyapunov function conditions, for deriving
$W_1I$ or more generally $T_\Phi I$ where $\Phi=\{(u,u); |u|\le
\phi\}$ with $\phi$ unbounded, and log-Sobolev inequality. To state properly the Lyapunov
function condition, it is necessary to enlarge the domain of the
generator. In this section,
 the Markov process
$((X_t), \pp_\mu)$ is reversible and its sample paths are
$\pp_\mu$-c\`adl\`ag (possibly with jumps).

A continuous function $h$ is said to be in the $\mu$-extended
domain $\dd_e(\LL)$ of the generator of the Markov process
$((X_t), \pp_\mu)$ if there is some measurable function $g$ such
that $\int_0^t |g|(X_s)\,ds<+\infty, \pp_\mu$-a.s.\! and
$$
M_t(h):=h(X_t) - h(X_0)-\int_0^t g(X_s)ds
$$
is a local $\pp_\mu$-martingale. It is obvious that $g$ is
uniquely determined  up to $\mu$-equivalence. In such case one
writes $h\in\dd_e(\LL)$ and $\LL h=g$.

The Lyapunov condition can now be stated:
\begin{itemize}
\item[$(H)$] There exist a continuous function $U:\mathcal{X}\to
[1,+\infty)$ in $\dd_e(\LL)$, a nonnegative function $\phi$ and a
constant $b>0$ such that
$$- \frac{\mathcal{L}U}{U}\ge \phi -b, \ \mu\textrm{-a.s.}$$
\end{itemize}
When the process is irreducible and the constant $b$ is replaced
by $b1_C$ for some ``small set" $C$, then it is well-known that
the existence of a positive  bounded $\phi$  such that
$\inf_{\XX\setminus C}\phi>0$ in $(H)$ is equivalent to Poincar\'e
inequality (see \cite{p-BCG}, for instance).

Lyapunov conditions are widely used to study the speed of
convergence of Markov chains \cite{MT} or Markov processes
\cite{p-DFG}, large or moderate deviations and essential spectral
radii \cite{Wu01, Wu04}. More recently, they have been used to
study functional inequalities as weak Poincar\'e inequality
\cite{p-BCG} or super-Poincar\'e inequality \cite{CGWW}. See Wang
 \cite{Wang} on weak and super Poincar\'e inequalities.

 \begin{thm}\label{thm51}
Assume that $\mu$ satisfies a Poincar\'e inequality with best
constant $\SG(\mu)<\infty$ and that the Lyapunov condition $(H)$
holds. Suppose moreover that $\phi\in L^2(\mu),$ that is
$\|\phi\|_2:=(\int\phi^2\,d\mu)^{1/2}<\infty.$ Then, for any $a\ge
2$ and for every probability measure $\nu$,
\begin{equation}\label{thm51a}
\|\phi(\nu-\mu)\|_{TV}\le\left(1+2b \SG(\mu)\right)\frac{a+1}{a-1}
I(\nu|\mu) +a\sqrt{2} \|\phi\|_2 \sqrt{\SG(\mu)\,I(\nu|\mu)}
\end{equation}
 and
\begin{equation}\label{thm51b}\|\sqrt{\phi}(\nu-\mu)\|_{TV}^2\le 2\left[3\left(1+2b\SG(\mu) \right)
 +2\sqrt{2 } \|\phi\|_2 \SG(\mu)
 \right] I(\nu|\mu)
\end{equation}
\end{thm}

\brmk {\rm Since $\|\phi(\nu-\mu)\|_{TV}=\sup_{u: |u|\le \phi}
\int u\,d(\nu-\mu)$, the inequalities in this theorem may be
regarded as $T_\Phi I$ in Theorem \ref{thm-a2.1} with
$\Phi=\{(u,u); u\in b\BB, |u|\le \phi\}$. Since

$$
W_1(\nu, \mu)=\sup_{f:\|f\|_{\mathrm{Lip}}\le 1} \int
f\,d[\nu-\mu]\le \inf_{x_0\in\XX} \|d(\cdot,x_0) (\nu-\mu)\|_{TV},
$$
one sees that (\ref{thm51b}) implies $W_1I(c)$ in Corollary
\ref{cor21} as soon as  $d(\cdot,x_0)\le C \sqrt{\phi}$ for some
$x_0\in \XX$ and $C>0$. As will be seen with the
Ornstein-Uhlenbeck process at Example \ref{exam51}, the order of
this inequality is sharp.} \nrmk

\begin{thm}\label{thm52} In the framework of Proposition \ref{prop25}, assume that the Bakry-Emery's curvature of $\mu=e^{-V} dx/Z$ is bounded from below by some constant $K\le 0$. 
Assume that $\mu$ satisfies a Poincar\'e inequality with best
constant $\SG(\mu)<\infty$. 

If the Lyapunov condition $(H)$
holds with $\phi(x)=c d(x,x_0)^2$ where $c>0$ and $x_0\in \XX$ is some fixed point, then $\mu$ satisfies the log-Sobolev inequality on the Riemannian manifold $\XX$.  
\end{thm}

Their proofs are based on the following  large deviation result.

\begin{lem}\label{lem52} For every continuous  function $U\ge 1$ in $\dd_e(\LL)$ such
that $-\LL U/U$ is $\mu$-a.e. lower bounded,
\begin{equation}\label{lem52a}
\int -\frac{\LL U}{U} g^2 \,d\mu \le \EE(g,g), \ \forall
g\in\dd(\EE).
\end{equation}
\end{lem}
When $U$ is bounded, this is contained in Deuschel-Stroock
\cite[Lemme 4.2.35]{DS}. \bprf   For any initial law $\beta,$
$$
N_t = U(X_t)\exp\left(-\int_0^t \frac{\LL U}{U}(X_s)ds\right)
$$
is a local $\pp_\beta$-martingale. Indeed, denoting
$A_t:=\exp\left(-\int_0^t \frac{\LL U}{U}(X_s)ds\right)$, It\^o's
formula is $ dN_t =A_t\, [dM_t(U)+ \LL U(X_t)\, dt] -\frac{\LL
U}{U}(X_t) A_t U(X_t)\,dt=A_t\, dM_t(U)$ where $M(U)$ is a local
$\pp_\beta$-martingale.  As $(N_t)$ is nonnegative, it is also a
$\pp_\beta$-supermartingale. Choosing $\beta:=U^{-1}\,\mu/Z$ with
$0<Z=\mu(U^{-1})\le1,$ one sees that for all $t\ge0$
$$
\ee_ {\beta} \exp\left(-\int_0^t \frac{\LL U}{U}(X_s)ds\right)\le
\ee_ \beta N_t \le \beta(U)=1/Z<+\infty.
$$
Let $u_n:=\min\{-\LL U/U, n\}$. The previous estimation implies
that
$$
F(u_n):=\limsup_{t\to\infty} \frac 1t\log
\ee_{\beta}\exp\left(\int_0^t u_n(X_s)ds\right)\le0.
$$
On the other hand by the lower bound of large deviation in
\cite[Theorem B.1, Corollary B.11]{Wu00b} and  Laplace-Varadhan
principle, as in the proof of $(c')\Rightarrow(a)$ in Theorem
\ref{thm-a2.1},
$$
F(u_n)\ge \sup\{\nu(u_n)-I(\nu|\mu);\ \nu\in M_1(E)\}.
$$
Thus $\int u_n d\nu\le I(\nu|\mu)$, which yields to (by letting
$n\to\infty$ and monotone convergence)

\begin{equation}\label{lem52b}
\int -\frac {\LL U} U\, d\nu \le I(\nu|\mu),\ \forall \nu\in
M_1(E).
\end{equation}
This is equivalent to (\ref{lem52a}) by the fact that $\EE(|h|,
|h|)\le \EE(h,h)$ for all $h\in\dd(\EE)$.

Note that one was allowed to apply the large deviation lower bound
\cite[Theorem B.1]{Wu00b} under $\pp_\beta$ since $\beta$ is
absolutely continuous with respect to $\mu$. In addition, in the
symmetric case, \cite[Corollary B.11]{Wu00b} states that the large
deviation rate function is $I(\cdot|\mu);$ it doesn't depend on
$\beta$ under the underlying  assumption that $\pp_\mu$ is
ergodic. As this lower bound holds for the topology of probability
measures weakened by all bounded measurable test functions
(sometimes called $\tau$-topology), one can apply the
Laplace-Varadhan principle to the continuous bounded function
$\nu\mapsto \nu(u_n).$
 \nprf

\bprf[Proof of Theorem \ref{thm52}] It is a combination of the Lyapunov function condition and 
the HWI inequality of Otto-Villani. 

We begin with the following fact (\cite[Proposition 7.10]{Vill03}):
$$
W_2^2(\nu,\mu) \le 2 \|d(\cdot, x_0)^2(\nu-\mu)\|_{TV}.
$$
Now for every function $g$ with $|g|\le \phi(x):=c d(x,x_0)^2$, we have by $(H)$,

$$
\aligned
\int g d(\nu-\mu) &\le \nu(\phi) + \mu(\phi)\\
&\le \int \left(- \frac{\LL U}{U} + b\right) d\nu + \mu(\phi)\\
&\le I(\nu|\mu) + b + \mu(\phi)
\endaligned
$$
where the last inequality follows by Lemma \ref{lem52}. Taking the supremum over all such $g$, we get 

$$
\frac {c}2 W_2^2(\nu,\mu) \le c \|d(\cdot, x_0)^2(\nu-\mu)\|_{TV} \le I(\nu|\mu) + b + \mu(\phi),
$$
which yields to (by the inequality at the beginning)
$$
W_2^2(\nu,\mu) \le \frac 2c I(\nu|\mu) + \frac 2c[b + \mu(\phi)].
$$
Substituting it into the HWI inequality of Otto-Villani (\ref{HWI-OV}), we obtain (using $2ab \le a^2 + b^2$)

\begin{equation}\label{thm52b}
\aligned
H(\nu|\mu) 
&\le 2 \sqrt{\frac 2c I(\nu|\mu) + \frac 2c[b + \mu(\phi)]} \sqrt{I(\nu|\mu)} - \frac K2 \left(\frac 2c I(\nu|\mu) + \frac 2c[b + \mu(\phi)]\right)\\
&\le A I(\nu|\mu) +B
\endaligned
\end{equation}
where 
$$A=(1-\frac K2)\frac 2c + 1,\ \ B=\frac 2c[b + \mu(\phi)](1-\frac K2). $$
Finally by Rothaus' lemma the non-tight log-Sobolev inequality (\ref{thm52b}) together with the spactral gap implies the tight log-Sobolev inequality
$$
H(\nu|\mu) \le [A + (B+2)\SG(\mu)] I(\nu|\mu).
$$
\nprf

\brmk{\rm In the case that the Bakry-Emery's curvature is bounded from below by a negative constant $K$, Wang's criterion \cite{Wang97} says that the log-Sobolev inequality holds if $\int e^{\lambda d^2(x,x_0)} d\mu(x)<+\infty$ for some $\lambda>|K|$. Our Lyapunov condition $(H)$ above is complementary to that result and sharp in order (as seen for $V(x)=|x|^\alpha$ on $\XX=\rr^d$). Furthermore our proof here is completely different and gives an explicit estimate of the log-Sobolev constant. 
}
\nrmk

 \bprf[Proof of Theorem \ref{thm51}] We may assume that $\nu=f\mu$ with $\sqrt{f}\in\dd(\EE)$ (trivial otherwise). For any $a\ge 2$, define $h:\rr\to
 \rr^+$ by

 $$
h(t)=\begin{cases} 0 &\text{if }\ t\le 1;\\
\sqrt{\frac{a+1}{a-1}}(t-1) &\text{if }\ t\in [1,a];\\
\sqrt{t^2-1} &\text{if }\ t\ge a.
\end{cases}
 $$
It is easy to see that $\|h\|_{\mathrm{Lip}}\le
\sqrt{\frac{a+1}{a-1}}$. Decompose
$$
\|\phi(\nu-\mu)\|_{TV}=\int \phi |f-1| d\mu= \int \phi
h^2(\sqrt{f}) d\mu + \int \phi[ |f-1|-h^2(\sqrt{f})] d\mu.
$$
First consider the last term. Since $t^2-1-h^2(t)\le a(t-1)$ for
$t\in [1,a]$, and $=0$ for $t\ge a\ge2$,

$$
\aligned \int \phi[ |f-1|-h^2(\sqrt{f})] d\mu &=\int \phi[1_{\{
f\le 1\}}  (1-f) + 1_{\{ 1\le f\le a^2\}} a(\sqrt{f}-1)
]d\mu\\
&\le a\int \phi |1-\sqrt{f}| d\mu
\endaligned
$$
which is not greater than
$$
\aligned a\|\phi\|_2 \|1-\sqrt{f}\|_2 &=a \|\phi\|_2\sqrt{2}
\sqrt{1-\mu(\sqrt{f})} \le a  \|\phi\|_2
\sqrt{2\Var_\mu(\sqrt{f})}\\
 &\le a\sqrt{ 2 \SG(\mu)}  \|\phi\|_2
\sqrt{I(\nu|\mu)}.
\endaligned
$$
We turn now to bound the crucial first term by means of
(\ref{lem52a}):
$$
\aligned \int \phi h^2(\sqrt{f}) d\mu &\le \int
\left(-\frac{\LL U}{U} + b \right)h^2(\sqrt{f}) d\mu \\
&\le \EE(h(\sqrt{f}), h(\sqrt{f})) + b \|h\|_{\mathrm{Lip}}^2\int (\sqrt{f}-1)^2 d\mu\\
&\le \|h\|_{\mathrm{Lip}}^2\EE(\sqrt{f}, \sqrt{f}) + 2b
\|h\|_{\mathrm{Lip}}^2
\Var_\mu(\sqrt{f})\\
&\le \left(1+2b\SG(\mu) \right)\frac{a+1}{a-1} I(\nu|\mu).
\endaligned
$$
Substituting these two estimates into our previous decomposition,
we obtain (\ref{thm51a}).
\\
For (\ref{thm51b}), noting that with Theorem \ref{thm2}: $\int
|f-1|d\mu\le 2\min\{1, \sqrt{\SG(\mu)I(\nu|\mu)}\},$ we have by
Cauchy-Schwarz inequality and (\ref{thm51a})
\begin{eqnarray*}
&&\|\sqrt{\phi}(\nu-\mu)\|_{TV}^2\\
 &\le& \int|f-1|d\mu \int\phi|f-1|d\mu \\
 &\le& 2\min\left(1, \sqrt{\SG(\mu)I(\nu|\mu)}\right)
    \left[\left(1+2b \SG(\mu)\right)\frac{a+1}{a-1}
    I(\nu|\mu) +a\sqrt{2} \|\phi\|_2
    \sqrt{\SG(\mu)I(\nu|\mu)}\right]\\
    &\le &2\left[\left(1+{2b}\SG(\mu)
    \right)\frac{a+1}{a-1}+a\sqrt{2} \|\phi\|_2\SG(\mu)
     \right] I(\nu|\mu)
\end{eqnarray*}
which gives  (\ref{thm51b}) with $a=2$.
 \nprf

\brmk {\rm When $\sqrt{2}(\SG(\mu)^{-1}+2b)\ge \|\phi\|_2$,
optimizing $a\ge 2$ in the proof of (\ref{thm51b}) above, we get
the slightly better inequality:
$$
\|\sqrt{\phi}(\nu-\mu)\|_{TV}^2 \le
2\left(\sqrt{2}\|\phi\|_2\SG(\mu)+(1+2b\SG(\mu)) + 2^{5/4}
\sqrt{\|\phi\|_2(1+2b\SG(\mu))\SG(\mu)} \right)\,I(\nu|\mu).
$$
Notice that by Lemma \ref{lem52} and condition (H), $b\ge
\mu(\phi)$ (in practice $b$ is much bigger). } \nrmk

From now on the positive constant $C$ may change from one place to
another.
\\
One can do some variation of the proof of (\ref{thm51b}) above.
For every $p>1$ and its conjugate number $q=p/(p-1)$, instead of
Cauchy-Schwarz we apply H\"older inequality to get for $\nu=f\mu$,
$$
\aligned \|\phi^{1/p}(\nu-\mu)\|_{TV}&\le \left(\int
|f-1|d\mu\right)^{1/q}\left(\int \phi |f-1|d\mu\right)^{1/p}
 \\
& \le 2^{1/q} \min\{1, \sqrt{\SG(\mu)I(\nu|\mu)}^{1/q}\}
(C_1I(\nu|\mu) +
 C_2\sqrt{I(\nu|\mu)})^{1/p}\\
&\le C [(1+ I(\nu|\mu))^{2/p}-1]^{1/2}.
 \endaligned
$$
 In other words, we have proved

 \bcor\label{cor52} Under the conditions of Theorem \ref{thm51},
 for any $p>1$, there exists some constant $\kappa>0$ such that
 for $\alpha(r)=\kappa [(1+r^2)^{p/2}-1]$,

\begin{equation}\label{thm51c}
\alpha\left( \|\phi^{1/p}(\nu-\mu)\|_{TV}\right) \le I(\nu|\mu),\
\forall \nu\in M_1(\XX).
\end{equation}
 \ncor

 \bcor\label{cor53} Let $\mu=e^{-V} dx/Z$ be a probability
 measure where
 $V\in C^\infty(\XX)$ is bounded from below and $|\nabla V|^2\in L^2(\mu).$ Let $\LL=\Delta -\nabla V\cdot
\nabla$ be the generator of the diffusion $(X_t)$ on the
non-compact connected complete  Riemannian manifold $\XX$. Assume
that for some $p>1,$
$$d(x,x_0)\le C (1
+|\nabla V|^2(x))^{1/p},\ \forall x\in\XX$$ and
$$
\gamma:=\limsup_{d(x, x_0)\to  \infty} \frac{\Delta V(x)}{|\nabla
V|^2(x)} <1.
$$
Then w.r.t. the Riemannian metric $d$, there exists $\kappa>0$
such that with $\alpha (r)=\kappa
    [(1+r^2)^{p/2}-1]$,
    \begin{equation}\label{cor53b}
\alpha(W_1(\nu,\mu)) \le I(\nu|\mu), \ \forall \nu\in M_1(\XX).
    \end{equation}
    In particular for every Lipschitz function $u$ with $\|u\|_{\mathrm{Lip}}\le
    1$ and any initial law $\beta$ with $d\beta/d\mu\in L^2(\mu),$
    \begin{equation}\label{cor53c}
\pp_\beta\left(\frac 1t \int_0^t u(X_s)ds > \mu(u)+r\right)\le
\left\|\frac{d\beta}{d\mu}\right\|_2 \exp\left(-t\kappa
    [(1+r^2)^{p/2}-1] \right), \ \forall t, r>0.
    \end{equation}
 \ncor

\bprf Let $\gamma'\in (\gamma,1)$ and $\lambda, \delta\in (0,1)$
sufficiently small so that $\lambda-\lambda^2>\gamma'  \lambda+
\delta$. For $U=e^{\lambda V},$ we have
$$
-\frac {\LL U}{U} = -\lambda \LL V - \lambda^2|\nabla V|^2=
(\lambda -\lambda^2) |\nabla V|^2 -\lambda \Delta V\ge \delta (1+
|\nabla V|^2)-b
$$
where $b:=\delta +\sup_\XX \left(\lambda \Delta V - \gamma'\lambda
|\nabla V|^2\right)$ is finite under our assumptions. Thus (H) is
satisfied with $\phi=\delta (1+|\nabla V|^2)$ which is in
$L^2(\mu)$ under our assumptions. On the other hand our
assumptions imply that $\phi$ tends to infinity at infinity. Hence
$(1-\LL)^{-1}$ is compact on $L^2(\mu)$ and $\SG(\mu)<\infty.$
Noting that
$$
W_1(\nu, \mu)\le \inf_{x_0\in\XX} \|d(\cdot,x_0) (\nu-\mu)\|_{TV},
$$
the statement now follows directly from Theorem \ref{thm51} and
Corollary \ref{cor52}.
 \nprf

\bexa\label{exam51} {\rm Let $\XX=\rr^n$, $V(x)=C |x|^\beta$ for
$|x|>1$ where $\beta\ge1, C>0$. Then (H) is satisfied for
$\phi=\delta (1+|\nabla V|^2)\sim C |x|^{2(\beta-1)}$ (when $|x|$
large).

\begin{enumerate}[(i)]
    \item If $\beta>3/2$, then the condition in Corollary
    \ref{cor53} is verified with $p=2(\beta-1)>1$, so we have
    (\ref{cor53c}) for
    Lipschitz observable $u$ with $p=2(\beta -1)$. Then we have Gaussian behavior for
    small  $r$, and even a super-Gaussian tail for large  $r$
    whenever $\beta>2$.

    \item Let $\beta\in [1, 3/2]$. Then for
    $\psi=(1+|x|)^{\beta-1}$, we have by Theorem
    \ref{thm51}(\ref{thm51b}),
    $$
\|\psi (\nu-\mu)\|_{TV}^2 \le C I(\nu|\mu).
    $$
Then the Gaussian deviation inequality holds true for the
observable $u$ satisfying $|u|\le C(1+|x|)^{\beta-1}$.

\item If $\beta=2$ (Ornstein-Uhlenbeck process), the inequality
(\ref{thm51b}) for $\sqrt{\phi}\sim C |x|$ (proved in (i)) becomes
the correct one in order: indeed if $\psi(x)\gg |x|$ at infinity
with $\mu(\psi)<+\infty$, one cannot hope that
$$
\|\psi(\nu-\mu)\|_{TV}^2\le C I(\nu|\mu),\ \forall \nu
$$
since by Theorem \ref{thm-a2.1}, this would imply that
$$
\ee_{\mu} \exp\left(\lambda \int_0^1 \psi(X_s) ds\right)\le
e^{\lambda \mu(\psi)+ C\lambda^2/2}, \ \forall \lambda\in \rr
$$
which yields, integrating w.r.t. $N(0,\sigma^2)(d\lambda)$ with
variance $\sigma^2<1/C$,
$$
\ee_{\mu} \exp\left(\delta \left(\int_0^1 \psi(X_s)
ds\right)^2\right)<+\infty
$$
for some $\delta>0.$ But this is impossible.
\end{enumerate}
} \nexa

We conclude by an example of jump process.

\bexa{\rm ($M/M/\infty$ queue). In this example $\XX=\nn$, $\mu$
is the Poisson measure with mean $\lambda>0$ and the Dirichlet
form is
$$
\EE(h,h) = \sum_{n\in \nn} (h(n+1)-h(n))^2 \mu(n)
$$
The associated generator is
$$
\LL h(n)= \lambda (h(n+1)-h(n)) + n(h(n-1)-h(n)), \ \forall n\ge0
$$
(with the convention that $h(-1)=h(0)$). Let $U(n)=e^{c n}$ where
$c>0$. We have
$$
-\frac{\LL U}{U} (n) = n (1-e^{-c}) -(e^c-1).
$$
Thus condition (H) is satisfied, and we have by Theorem
\ref{thm51} that for $\psi(n):=\sqrt{1+n}$,
$$
\|\psi (\nu-\mu)\|_{TV}^2 \le C I(\nu|\mu), \ \forall \nu\in
M_1(\nn).
$$
By Theorem \ref{thm-a2.1}, this gives the Gaussian deviation
inequality for any observable $u$ so that $|u(n)|\le C
\sqrt{1+n}$. See Joulin \cite{JouPhD} and Liu-Ma \cite{p-LiuMa06}
for previous studies on deviation inequalities of this model. Note
that they only obtain Poisson tail by their approach for the same
test function. Remark also that our result provides exponential
tail for $u(n)=n$, which is close of the conjectured Poisson
behavior.
 }\nexa

%\bibliographystyle{plain}
%\bibliography{bib-christian}

\end{document}